\documentclass[12pt]{amsart}

\usepackage{amssymb}
\usepackage[normalem]{ulem}
\usepackage[colorlinks=true]{hyperref}
\usepackage{mathrsfs}
\usepackage[all]{xy} 
\usepackage{setspace}

\usepackage{enumitem}

\usepackage{algorithmic}
\usepackage[ruled]{algorithm}
\usepackage{caption}

\newcommand{\supp}{\mathrm{supp}}
\newcommand{\PP}[1]{\mathcal P_{#1}}
\newcommand{\RR}{\mathcal T_{\tilde I,\tilde J}}
\newcommand{\mult}{\mathcal{X}}
\newcommand{\Ht}{\mathrm{Ht}}
\newcommand{\NN}{\mathcal{N}}
\newcommand{\Nf}{\mathrm{Nf}}
\newcommand{\kk}{\mathbb{K}}
\newcommand{\T}{\mathbb{T}}
\newcommand{\QR}{S}
\newcommand{\MFFunctor}[1]{\underline{\mathbf{Mf}}_{#1}}
\newcommand{\MFScheme}[1]{\mathbf{Mf}_{#1}}
\newcommand{\HFunctor}[2]{\underline{\mathbf{Hilb}}_{#1}^{#2}}
\newcommand{\HScheme}[2]{\mathbf{Hilb}_{#1}^{#2}}
\newcommand{\Proj}{\mathrm{Proj}}
\newcommand{\Spec}{\mathrm{Spec}}
\newcommand{\rid}[1]{\longrightarrow_{#1}}
\newcommand{\crid}[1]{\longrightarrow_{#1}^+}
\newcommand{\PGL}[1]{\mathrm{PGL}({#1})}

\newtheorem{theorem}{Theorem}[section]
\newtheorem{corollary}[theorem]{Corollary}
\newtheorem{proposition}[theorem]{Proposition}
\newtheorem{lemma}[theorem]{Lemma}

\theoremstyle{definition}
\newtheorem{definition}[theorem]{Definition}

\newtheorem{example}[theorem]{Example}
\theoremstyle{remark}
\newtheorem{remark}[theorem]{Remark}
\numberwithin{equation}{section}

\begin{document}

\title[Hilbert schemes over quotient rings via relative marked bases]{Open Covers and Lex Points of Hilbert schemes over quotient rings via relative marked bases
}

\author[C.~Bertone]{Cristina Bertone}
\address{Dipartimento di Matematica \lq\lq G.~Peano\rq\rq\\ Universit\`a degli Studi di Torino\\ 
         Via Carlo Alberto 10\\ 10123 Torino\\ Italy.}
\email{\href{mailto:cristina.bertone@unito.it}{cristina.bertone@unito.it}}
\urladdr{\url{https://sites.google.com/view/cristinabertone/}}

\author[F.~Cioffi]{Francesca Cioffi}
\address{Dip. di Matematica e Appl. \\ Universit\`a degli Studi di Napoli Federico II\\ Via Cintia \\ 80126 Napoli \\ Italy.}
\email{\href{mailto:cioffifr@unina.it}{cioffifr@unina.it}}

\author[M.~Orth]{Matthias Orth}
\address{Institute of Mathematics\\ University of Kassel\\ 
         34109 Kassel, Germany.}
\email{\href{mailto:morth@mathematik.uni-kassel.de}{morth@mathematik.uni-kassel.de}}

\author[W.~Seiler]{Werner Seiler}
\address{Institute of Mathematics\\ University of Kassel\\ 
         34109 Kassel, Germany.}
\email{\href{mailto:seiler@mathematik.uni-kassel.de}{seiler@mathematik.uni-kassel.de}}
\urladdr{\url{http://www.mathematik.uni-kassel.de/~seiler/}}

\subjclass[2020]{13P10, 14C05, 14M05, 13H10, 13F55}
\keywords{Marked basis, Hilbert scheme, Cohen-Macaulay ring, Macaulay-Lex ring, open cover, lex-point}

\begin{abstract}
We introduce the notion of a {\em relative marked basis} over quasi-stable ideals, together with constructive methods and a functorial interpretation, developing computational methods for the study of Hilbert schemes over quotients of polynomial rings. Then we focus on two applications.  

The first has a theoretical flavour and produces an explicit open cover of the Hilbert scheme when the quotient ring is Cohen-Macaulay on quasi-stable ideals. Together with relative marked bases, we use suitable {\em general} changes of variables which preserve the structure of the quasi-stable ideal, against the expectations.

The second application has a computational flavour. When the quotient rings are Macaulay-Lex on quasi-stable ideals, we investigate the lex-point of the Hilbert schemes and find examples of both smooth and singular lex-points. 
%These examples answer an open problem recently posed in literature.
\end{abstract}

\maketitle

\section*{Introduction}

Let $R=\kk[x_0,\dots,x_n]$ be the polynomial ring over a field $\kk$ in $n+1$ variables, endowed with the order  $x_0<\dots<x_n$, and  $I$ be an ideal of $R$. %Let $\HScheme{X}{p(z)}$ denote the Hilbert scheme over $X=\Proj(R/I)\subset \mathbb P^n_{\kk}$.

We provide a way to analyse Hilbert schemes over a quotient ring $R/I$ using a computer algebra system. The tools that we develop are based on the theory of marked bases over quasi-stable ideals, together with their properties and functorial features \cite{CMR2015,BCRAffine,Quot}. However, the different setting that is considered in this paper presents new problems to solve. 

We apply our tools to achieve two different tasks under the hypothesis that $I$ is a monomial quasi-stable ideal. The first one concerns the study of a suitable open cover of such a Hilbert scheme when $R/I$ is Cohen-Macaulay. The second one regards the study of lex-points when $R/I$ is a Macaulay-Lex ring, i.e.~a ring in which an analogue of Macaulay's Theorem characterizing the Hilbert functions of homogeneous ideals in a polynomial ring holds. 

We start giving a first non-obvious insight in the application of marked bases to quotient polynomial rings (Section \ref{sec:MFandHilb}). Then we push forward our investigation and look for results analogous to those for relative Gr\"obner bases and relative involutive bases of ideals in quotient rings that have been recently developed in \cite{HOS}. Hence, we define {\em relative marked bases}, which turn out to be suitable to work in $R/I$, in particular when $I$ is a quasi-stable ideal, and study their functorial interpretation (Sections \ref{sec:relative marked bases} and \ref{sec:relative marked functor}). 

An immediate consequence of this study is that we can construct some open subschemes of the Hilbert scheme over $R/I$ by means of relative marked bases, too (see Proposition \ref{prop:sottoschema} and Theorem \ref{th:open subfunctor}).

When the field $\kk$ is infinite and the quotient ring  on a quasi-stable ideal is Cohen-Macaulay, we develop this feature to describe an open cover of Hilbert schemes over such a quotient ring, which we obtain thanks to suitable changes of variables  applied on open subsets which parameterise relative marked bases (see Theorem~\ref{thm:ric}). This result is achieved 
%in Section \ref{sec:open cover}, 
generalising the method which is described in \cite{Quot} and which is based on deterministically computable suitable linear changes of variables. 

The novelty of this approach consists in the fact that we show that there are computable general linear changes of variables of the quotient ring that by definition preserve the complete structure of the ideal on which the quotient is performed, instead of destroying it, as it could be expected (e.g., see \cite[Introduction]{MP}). Hence, this result is not obvious and, together with the underlying idea, is new in the context of the present paper.

Even when $\kk$ is not infinite or the quotient ring is not Cohen-Macaulay, the availability of  open subschemes described by means of the relative marked bases  encourages the study of local properties.  For example, when the quotient ring is Macaulay-Lex%on a quasi-stable ideal
, the explicit computation of relative marked schemes can be useful in the investigation of the properties of the lex-point of Hilbert schemes over such quotient rings. 

It is indeed very well-known that every non-empty Hilbert scheme over a polynomial ring on a field has a unique point, called the {\em lex-point}, that is defined by a lex-ideal, and which is smooth (see \cite{ReSti}) and  characterized by the property that its defining saturated lex-ideal has the minimal possible Hilbert function among the points of the same Hilbert scheme. 

It is even true that every non-empty Hilbert scheme over a Macaulay-Lex ring on a quasi-stable ideal has the lex-point, which moreover has the minimal Hilbert function (Theorem \ref{th:lex-point}). However, in Section \ref{sec:lex-point} we give classes of examples both of smooth and singular lex-points in Macaulay-Lex rings over quasi-stable ideals. 

The problem of the smoothness of the lex-point  is %highlighted in a special case of Macaulay-Lex rings in \cite[Remark~1.6]{CS2} and
 studied also
 in other Hilbert schemes, see for instance \cite{RS-2022}. In fact, % It is noteworthy that
  it is not possible to extend the proof for the smoothness of the lex-point of the Hilbert scheme $\HScheme{\mathbb P^n}{p(z)}$ given in \cite{ReSti} to other cases because the Zariski tangent space is not the same (see \cite[Section 1.7]{GMP} and \cite[Proposition~2.1]{PSFlips}).

To the best of our knowledge, analogous examples are not yet available in the literature. Moreover, the benefits obtained by the use of relative marked bases are evident when we count the number of parameters involved in the computations (see Remark~\ref{rem:number parameters}), as we highlight throughout the descriptions of some of our examples.

\section{Preliminaries}
\label{sec:preliminaries}

Let $\kk$ be a field and $A$ any Noetherian $\kk$-algebra with $1_A=1_{\kk}$. Take the polynomial ring $R=\kk[x_0,\dots,x_n]$ endowed with the order $x_0<\cdots<x_n$, and $R_A:=R\otimes_{A}A=A[x_0,\dots,x_n]$, so that $R=R_{\kk}$. A \emph{term} is a power product $x^\alpha= x_0^ {\alpha_0}\cdots x_n^{\alpha_n}$. 
We denote by $\T$ the set of terms. For every $x^\alpha \in \T$, we denote by $\min(x^\alpha)$ the smallest variable dividing $x^\alpha$ and by $\mult(x^\alpha)$ the set of the variables smaller than or equal to $\min (x^\alpha)$ which is called the set of {\em multiplicative variables} of $x^\alpha$. If $N$ is a finite set of polynomials, we denote by $\langle N\rangle_A$ the $A$-module generated by $N$, and by $(N)$ the ideal generated by $N$ in $R_A$. 

We use the standard grading on $R_A$, that is $\deg(x_j)=1$ for all $j\in \{0,\dots,n\}$ and $\deg(a)=0$ for all $a\in A$. Hence we have $\deg(x^\alpha)=\vert \alpha\vert=\sum \alpha_i$.
We assume that the polynomials, the ideals and $A$-modules involved in our definitions, statements and arguments are {\em homogeneous} with respect to this standard grading on $R_A$. 

For an ideal $I$, we denote by $I_t$ the vector space of the homogeneous polynomial of $I$ of a given degree $t$ and set $I_{\geq t}:=\bigoplus_{s\geq t} I_s$.

When we write an equality of the kind $I=B_1\oplus B_2$, where $I$ is an ideal and $B_1, B_2$ are $A$-modules or ideals, we also mean $I_s=(B_1)_s\oplus (B_2)_s$ for every $s\geq 0$. In such situations, we will say that the equality is \emph{graded}.

An ideal $\tilde I$ is \emph{monomial} if it is generated by a set of terms. A monomial ideal $\tilde I$ has a unique minimal set of generators consisting of terms and we call it the \emph{monomial basis} of $\tilde I$, denoted by $\mathcal B_{\tilde I}$. We define $\NN(\tilde I)\subseteq \T$ as the set of terms in $\T$ not belonging to~$\tilde I$.  For every polynomial $f\in R_A$, $\supp(f)$ is the set of terms appearing in $f$ with a non-zero coefficient.
For every polynomial $f \in R_{A}$, an {\em $x$-coefficient} of $f$ is the coefficient in $A$ of a term in $\T\cap \supp(f)$. 

\begin{definition}\label{def:qsIdeal}
For every $x^\alpha $  in $\T$, we define the \emph{Pommaret cone} of $x^\alpha$ as
\[
 \mathcal C_{\mathcal P}(x^\alpha):=\{x^\delta x^\alpha \mid \delta_i=0, \, \forall x_i\notin \mult(x^\alpha) \}\subset \T.
\]
A finite set $U$ of terms generating an ideal $\tilde I$ is called a \emph{Pommaret basis} of $\tilde I$ if
\begin{equation}\label{eq:decPJ}
\tilde I\cap \T=\bigsqcup_{x^\alpha \in U}\mathcal C_{\mathcal P}(x^\alpha).
\end{equation}
A \emph{quasi-stable} ideal is a monomial ideal having a Pommaret basis. If $\tilde I$ is a quasi-stable ideal, we denote by $\PP{\tilde I}$ its Pommaret basis.
\end{definition}

\begin{proposition}[{\cite[Theorem 9.2]{Seiler2009II}, \cite[Theorem 5.5.15]{Seiler:InvolutionBook}}]
    If $\tilde I$ is a quasi-stable ideal, then $\mathcal B_{\tilde I}\subseteq \PP{\tilde I}$ holds and the regularity $\mathrm{reg}(\tilde I)$ coincides with the maximum degree of a term in $\PP{\tilde I}$.
\end{proposition}

\section{Marked bases and marked functors} 
\label{sec:MFandHilb}

\subsection{Marked bases} A \emph{marked polynomial} is a polynomial $f\in R_A$ together with a fixed term $x^\alpha\in\supp(f)$ whose coefficient is equal to $1_A$ (see \cite{RS}).
This term is called \emph{head term} of $f$ and denoted by $\Ht(f)$.
%If $f\in R_A$ is a marked polynomial, we set $\mult(f) := \mult(\Ht(f))$ and call  it  the set of  \emph{multiplicative variables of $f$}.

From now, let $\tilde I$ denote a quasi-stable ideal.

\begin{definition}\cite[Definition 5.1]{CMR2015}\label{def:MarkedSet}
A \emph{$\PP{\tilde I}$-marked set} is a finite set $F \subset R_A$ of exactly $\vert \PP{\tilde I}\vert$ marked homogeneous polynomials $f_\alpha$ with pairwise distinct head terms $\Ht(f_\alpha) = x^\alpha \in \PP{\tilde I}$ and $\supp(f_\alpha - x^\alpha) \subset \langle \NN (\tilde I) \rangle_A$. A $\PP{\tilde I}$-marked set $F$ is a \emph{$\PP{\tilde I}$-marked basis} of the ideal $(F)$ if the graded decomposition $\left(R_A\right) = (F) \oplus \langle \mathcal {N}(\tilde I) \rangle_A$ holds.
\end{definition}

\begin{definition}\cite[Definition 5.3]{CMR2015}\label{relazione omogenea}
Given a $\PP{\tilde I}$-marked set $F=\{f_\alpha\}_{x^\alpha\in\PP{\tilde I}}$, the set $F^{\ast}:=\{x^\eta f_\alpha  \ \vert \ x^\eta x^\alpha  \in \mathcal C_{\mathcal P}(x^\alpha) \}\subseteq (F)$ is made of homogeneous polynomials which are marked on the terms of $\tilde I$ in the natural way  $\Ht(x^\eta f_\alpha )=x^\eta \Ht(f_\alpha)$. 

We denote by $\rid{F^{\ast}}$ the reflexive and transitive closure of the following reduction relation on $R_A$: $f$ is in relation with $f'$ if $f'=f-\lambda x^\eta f_\alpha$, where $x^\eta f_\alpha\in F^{\ast}$ and $\lambda\neq 0_A$ is the coefficient of the term $x^{\eta+\alpha}$ in $f$.

We will write $f \crid{F^{\ast}} f_0$ if $f\in R_A$, $f \rid{F^{\ast}} f_0$ and $f_0\in \langle \NN(\tilde I)\rangle_A$. In this case we say that \lq\lq $f$ is reduced to $f_0$ by $F^\ast$\rq\rq, and that \lq\lq $f_0$ is reduced with respect to $F^\ast$\rq\rq.
\end{definition}

It is noteworthy that the reduction relation $\rid{F^\ast}$ is Noetherian and confluent (see \cite[Theorem 5.9 and and Corollary 5.11]{CMR}). 

\begin{lemma}\label{lem:subred}
Let $E$ be any subset of a $\PP{\tilde I}$-marked set $F$. 
Then, letting  $E^\ast=\{x^\eta p_\beta  \ \vert \ x^\eta x^\beta  \in \mathcal C_{\mathcal P}(x^\beta), p_\beta \in E \}$, the subreduction relation $\rid{E^\ast}$ is Noetherian and confluent.
\end{lemma}

\begin{proof} 
The reduction relation $\rid{E^\ast}$ is obviously defined as a subreduction of $\rid{F^\ast}$, like it is suggested in \cite[Definition 3.4]{CMR}.
Then, the reduction relation $\rid{E^\ast}$ is Noetherian because it is a subreduction of $\rid{F^\ast}$, which is Noetherian. Moreover, $\rid{E^\ast}$ is confluent  because it is Noetherian and has disjoint cones (see also \cite[Remark~7.2]{CMR}).
\end{proof}

\begin{remark}\label{rem:subred}
Observe that, in the hypotheses of Lemma \ref{lem:subred}, we write $p\crid{E^\ast}h$ when $\supp(h)$ is included in $\T\setminus\left(\cup_{x^\beta \in \{\Ht(p)\vert p\in E\}} \mathcal C_{\mathcal P}(x^\beta)\right)$, according to Definition \ref{relazione omogenea}.
\end{remark}

\subsection{Marked functors} It is possible to parameterise the set of ideals $I$ having a $\PP{\tilde I}$-marked basis by means of a functor from the category of Noetherian $\kk$-Algebras to that of Sets, which turns out to be represented by an affine scheme.
We briefly recall the definition of this functor and the construction of this affine scheme.
 
The {\em marked functor} from the category of Noetherian $\kk$-algebras to the category of sets
\[
\MFFunctor{\tilde I}: \underline{\text{Noeth}\ \kk\!\!-\!\!\text{Alg}} \longrightarrow \underline{\text{Sets}}
\]
associates to any Noetherian $\kk$-algebra $A$ the set
$$\MFFunctor{\tilde I}(A):=\{ (G) \subset R_A\mid G\text{ is a } \PP{\tilde I}\text{-marked basis}\}$$
and to any morphism of $\kk$-algebras $\sigma: A \rightarrow {A'}$ the map
\[
\renewcommand{\arraystretch}{1.3}
\begin{array}{rccc}
\MFFunctor{\tilde I}(\sigma):& \MFFunctor{\tilde I}(A) &\longrightarrow& \MFFunctor{\tilde I}({A'})\\
&(G) & \longmapsto& (\sigma(G))\;.
\end{array}
\]
Note that the image $\sigma(G)$ under this map is indeed again a
$\PP{\tilde I}$-marked basis, as we are applying the functor $-\otimes_A {A'}$
to the decomposition
$\left(R_A\right)_s =  (G)_s \oplus \langle
  \mathcal {N}(\tilde I)_s \rangle_A$  for every degree $s$.
  
\begin{remark}\label{rem:lemmaMP}
Generalising \cite[Proposition 2.1]{LR2} to quasi-stable ideals, we obtain
\[
\{ (G) \subset R_A\mid G\text{ is a } \PP{\tilde I}\text{-marked basis}\}=\{ I \subset R_A \text{ ideal } \mid  R_A=I\oplus \mathcal \langle \mathcal {N}(\tilde I)\rangle_A\}.
\]
\end{remark}

The functor $\MFFunctor{\tilde I}$ is represented by the affine scheme $\MFScheme{\tilde I}$ that can be explicitly constructed by the following procedure. We consider the $\kk$-algebra $\kk[C]$, where $C$ denotes the finite set of variables
$\bigl\{C_{\alpha\eta}\mid x^\alpha \in \PP{\tilde I}, x^\eta
\in \NN(\tilde I), \deg(x^\eta)=\deg(x^\alpha)\bigr\}$, and
construct the $\PP{\tilde I}$-marked set $\mathscr  G\subset R_{\kk[C]}$ consisting of the following marked polynomials
\begin{equation}\label{polymarkKC}
g_\alpha=x^\alpha-\sum_{x^{\eta}\in \NN(\tilde I)_{\vert\alpha\vert}}C_{\alpha\eta }x^\eta 
\end{equation}
with $x^{\alpha}\in\PP{\tilde I}$.
According to Definition \ref{relazione omogenea}, we consider
\[
\mathscr G^\ast =\{x^\delta g_\alpha \ \vert \ g_\alpha \in \mathscr G, x^\delta x^\alpha \in \mathcal C_{\mathcal P}(x^\alpha)\}.
\]
Then, by the Noetherian and the confluent reduction procedure given in Definition~\ref{relazione omogenea}, for every term $x^\alpha \in \PP{\tilde I}$ and every  variable $x_i$  not belonging to $\mult(x^\alpha)$, we compute a polynomial $p_{\alpha,i}\in \langle \NN(I)_{\vert\alpha\vert+1}\rangle_A$ such that $x_i g_\alpha-p_{\alpha,i}\in \langle \mathscr G^\ast \rangle_A$. We then denote by $\mathscr U$ the ideal generated in $\kk[C]$ by the $x$-coefficients of the polynomials~$p_{\alpha,i}$. Then, we have $\MFScheme{\tilde I}=\mathrm{Spec} (\kk[C]/\mathscr U)$  (\cite[Remark 6.3]{CMR2015},\cite[Theorem 5.1]{Quot}).

From now we assume that $\tilde I$ is in particular a {\em saturated} quasi-stable ideal. This implies that $x_0$ does not divide any term of $\mathcal B_{\tilde I}$ and $R/\tilde I$ has positive Krull dimension. Then $\tilde I_{\geq t}$ is quasi-stable too, for every integer $t$, so that we can consider $\MFScheme{\tilde I_{\geq t}}$. Let $\tilde p(z)$ be the Hilbert polynomial of $R/\tilde I$ and $\HScheme{\mathbb P^n}{\tilde p(z)}$ be the Hilbert scheme that parameterises the closed subschemes of $\mathbb P^n$ having Hilbert polynomial $\tilde p(z)$. Then, $\MFScheme{\tilde I_{\geq t}}$ embeds in $\HScheme{\mathbb P^n}{\tilde p(z)}$, for every integer $t$ (see \cite[Proposition 6.13]{BCRAffine}). 

\subsection{Parametrisation of saturated ideals in quotient rings} 
Let $I$ now be a {\em saturated} ideal with a $\PP{\tilde I}$-marked basis and take $X:=\Proj(R/I)$ together with its Hilbert polynomial $p_X(z)$, which is equal to $\tilde p(z)$.

Let $\tilde J$ be a {\em saturated} quasi-stable ideal containing $\tilde I$ and let $p(z)$ be the Hilbert polynomial of $R/\tilde J$. Then, $p(z)$ is {\em smaller} than $p_X(z)$, in the sense that $p(t)\leq p_X(t)$, for $t\gg 0$.
Let $\HScheme{X}{p(z)}$ be the Hilbert scheme that parameterises the closed subschemes of $X=\Proj(R/I)$ with Hilbert polynomial~$p(z)$ and represents the Hilbert functor $\HFunctor{X}{p(z)}$. 
 
If $\PP{\tilde J}$ does not contain any term divisible by $x_1$, we set $\rho_{\tilde J}:=1$. Otherwise, we set $\rho_{\tilde J}:=\max\{\deg(x^\alpha) \ \vert \ x^\alpha \in \PP{\tilde J} \text{ is divisible by } x_1\}$.

%\begin{remark}\label{rem:saturati}
%It is noteworthy that, if $\tilde J$\sout{$\subset R$} is a saturated quasi-stable ideal then, for every Noetherian $\kk$-algebra $A$ and integer $t$, an ideal $(G)$ belonging to} \sout{$\MFScheme{\tilde J_{\geq t}}(A)=\{ (G) \subset R_A\mid G\text{ is a } \PP{\tilde J_{\geq t}}\text{-marked basis}\}$ is of type $J_{\geq t}$, where $J$ is a saturated ideal (see \cite[Corollary 3.7]{BCRAffine}). More precisely, $J={J^{sat}}_{\geq t}$. Furthermore, $x_0$ is generic for $J$, meaning that $J^{sat}=(J:x_0^\infty)$ \cite[Theorem 3.5]{BCRAffine}.
%\end{remark}

\begin{proposition}\label{prop:sottoschema}
With the above notations, 
\begin{enumerate}
\item for every $t\geq \rho_{\tilde J}-1$, $\MFScheme{\tilde J_{\geq t}} \cap \HScheme{X}{p(z)} \cong \MFScheme{\tilde J_{\geq t+1}} \cap \HScheme{X}{p(z)}$;
\item for every $t\geq \rho_{\tilde J}-1$, $\MFScheme{\tilde J_{\geq t}} \cap \HScheme{X}{p(z)}$ is an open subscheme of $\HScheme{X}{p(z)}$.
\end{enumerate}
\end{proposition}

\begin{proof}
We recall that $\HScheme{X}{p(z)}$ is a (closed) subscheme of $\HScheme{\mathbb P^n}{p(z)}$ (e.g.~\cite[Exercise~VI-26]{EH}) and that  $\MFScheme{\tilde J_{\geq t}}$ is a (locally closed) subscheme of $\HScheme{X}{p(z)}$, for every integer $t$ (see \cite[Proposition 6.13(iii)]{BCRAffine}).
Then, the first item is a consequence of \cite[Corollary 6.11]{BCRAffine} which states $\MFScheme{\tilde J_{\geq t}} \cong \MFScheme{\tilde J_{\geq t+1}}$, for every $t\geq \rho_{\tilde J}-1$.
The second item is a consequence of the fact that, for every $t\geq \rho_{\tilde J}-1$, $\MFScheme{\tilde J_{\geq t}}$ is even an open subscheme of the Hilbert scheme $\HScheme{\mathbb P^n}{p(z)}$. Indeed, $\MFFunctor{\tilde J_{\geq t}}$ is an open subfunctor of $\HFunctor{X}{p(z)}$, like it is stated and proved in \cite[Proposition 6.13(ii)]{BCRAffine}.
\end{proof}

\begin{example}
If we consider the ideal $\tilde J:=(x_3^2,x_3x_2)$, then $\rho_{\tilde J}=1$, hence $\tilde J=\tilde J_{\geq \rho_{\tilde J}-1}$ and $\MFScheme{\tilde J} \cong \MFScheme{\tilde J_{\geq t}}$, for every $t\geq 0$. Starting from the following two marked polynomials:

$x_2x_3+x_0^2c_1 + x_0x_1c_2 + x_1^2c_3 + x_0x_2c_4 + x_1x_2c_5 + x_2^2c_6 + x_0x_3c_7 + x_1x_3c_8$,

$x_3^2+x_0^2c_9 + x_0x_1c_{10} + x_1^2c_{11} + x_0x_2c_{12} + x_1x_2c_{13} + x_2^2c_{14} + x_0x_3c_{15} + x_1x_3c_{16}$,

\noindent the defining ideal $\mathcal U  \subset \mathbb Q[c_1,\dots,c_{16}]$ of $\MFScheme{\tilde J}$ is computed obtaining

$\mathcal U=(
  c_1c_6c_7 - c_1c_4 - c_7c_9 + c_1c_{15},
  c_2c_6c_7 + c_1c_6c_8 - c_2c_4 - c_1c_5 - c_8c_9 - c_7c_{10} + c_2c_{15} + c_1c_{16},
  c_3c_6c_7 + c_2c_6c_8 - c_3c_4 - c_2c_5 - c_8c_10 - c_7c_{11} + c_3c_{15} + c_2c_{16},
  c_4c_6c_7 - c_4^2 - c_1c_6 - c_7c_{12} + c_4c_{15} - c_9,
  c_5c_6c_7 + c_4c_6c_8 - 2c_4c_5 - c_2c_6 - c_8c_{12} - c_7c_{13} + c_5c_{15} + c_4c_{16} - c_{10},
  c_6^2c_7 - 2c_4c_6 - c_7c_{14} + c_6c_{15} - c_{12},
  c_6c_7^2 - c_4c_7 + c_1,
  c_3c_6c_8 - c_3c_5 - c_8c_{11} + c_3c_{16},
  c_5c_6c_8 - c_5^2 - c_3c_6 - c_8c_{13} + c_5c_{16} - c_{11},
  c_6^2c_8 - 2c_5c_6 - c_8c_{14} + c_6c_{16} - c_{13},
  2c_6c_7c_8 - c_5c_7 - c_4c_8 + c_2,
  c_6c_8^2 - c_5c_8 + c_3,
  -c_6^2 - c_{14}).$

\noindent Although $\tilde J\not=\tilde J_{\geq 3}$ and the defining ideal $\mathcal U'$ of $\MFScheme{\tilde J_{\geq 3}}$ is contained in the ring $\mathbb Q[c_1,\dots, c_{91}]$ and so it is different from $\mathcal U$, the ring $ \mathbb Q[c_1,\dots,c_{16}]/\mathcal U$ is isomorphic to the ring $\mathbb Q[c_1,\dots,c_{91}]/\mathcal U'$ by \cite[Corollary 6.11]{BCRAffine}.   
If we consider instead the ideal $\tilde J:=(x_2^3,x_1)$ with Pommaret basis $\mathcal P_{\tilde J}=\{x_2^3, x_1, x_1x_2,x_1x_2^2\}$, then $\rho_{\tilde J}=1$ like before. Also in this case we have $\tilde J=\tilde J_{\geq \rho_{\tilde J}-1}$ and $\MFScheme{\tilde J} \cong \MFScheme{\tilde J_{\geq t}}$, for every $t\geq \rho_{\tilde J}-1 = 0$.
\end{example}

\begin{lemma}\label{lem:satEquiv}
With the notation above, let $F$ be any set of polynomials generating~$I$. If $G$ is a $\PP{\tilde J_{\geq t}}$-marked basis for some $t\geq \rho_{\tilde J}-1$, then the following statements are equivalent:
\begin{enumerate}[label=(\roman*)]
\item \label{it:satEquiv1}$(G)^{sat}\supseteq I$;
\item \label{it:satEquiv2}$(G)\supseteq I_{\geq t}$;
\item  \label{it:satEquiv3}$(G)\supseteq \{x_0^{\max\{0,t-\deg(f)\}}f\vert f\in F \}$.
\end{enumerate}
\end{lemma}

\begin{proof}
By \cite[Theorem 3,5 and Corollary 3.7]{BCRAffine} we have  $(G)={(G)^{sat}}_{\geq t}$ and $(G)^{sat}=((G):x_0^\infty)$. 
It is immediate that \ref{it:satEquiv1} implies \ref{it:satEquiv2} and \ref{it:satEquiv2} implies \ref{it:satEquiv3}. 
  
We now prove that item \ref{it:satEquiv3} implies item \ref{it:satEquiv1}.
By hypothesis, for every $f\in F$, we have that either $f$ belongs to $(G)\subseteq (G)^{sat}$, if $\deg(f)\geq t$, or $x_0^{t-\deg(f)}f$ belongs to $(G)$, if $\deg(f)<t$. In this latter case, $f\in ((G):x_0^\infty)=(G)^{sat}$. 
\end{proof}

Let $Z$ be a set of polynomials generating~$I$ (like, for example, the marked basis~$F$). For every $f\in Z$, we take an integer $d$ in the following way: if $\deg(f)\geq t$, then $d:=0$, otherwise $d:=t-\deg(f)$.
By a reduction relation like in Definition \ref{relazione omogenea}, we compute a polynomial $p_{f}\in \langle \NN(\tilde J_{\geq t})_{d}\rangle_A$  such that $x_0^d f-p_{f}\in \langle \mathscr G^\ast \rangle_A$. We then denote by $\mathscr V_Z$ the ideal generated in $\kk[C]$ by the $x$-coefficients of the polynomials~$p_{f}$.

\begin{theorem}\label{thm:constrRelScheme1}
With the above notations, $\MFScheme{\tilde J_{\geq t}} \cap \HScheme{X}{p(z)}$ is isomorphic to the affine scheme $\Spec(\kk[C]/(\mathscr U+\mathscr V_Z))$, for every $t\geq \rho_{\tilde J}-1$. 
\end{theorem}

\begin{proof}
Consider the $\PP{\tilde J_{\geq t}}$-marked set $\mathscr  G\subset R_{\kk[C]}$ made of the polynomials in \eqref{polymarkKC}. For every $\kk$-algebra~$A$, a $\PP{\tilde J}$-marked set in $R_A$ is uniquely and completely given by a $\kk$-algebra morphism $\varphi: \kk[C]\rightarrow A$ defined by $\varphi(C_{\alpha\gamma})=c_{\alpha\gamma}\in A$, for every $x^\alpha \in \PP{\tilde J_{\geq t}}, x^\gamma \in \NN(\tilde J_{\geq t})_{\vert\alpha\vert}$. We extend $\varphi$ to a morphism from $ R_{\kk[C]}$ to $R_A$ in the obvious way.

It is sufficient to observe  that $\varphi(\mathscr G)\subset R_A$ is a $\PP{\tilde J_{\geq t}}$-marked basis if and only if the generators of $\mathscr U$ vanish at $c_{\alpha\gamma}\in A$.  Furthermore, by Lemma \ref{lem:satEquiv}, the saturation of the ideal generated by $\varphi(\mathscr G)$ in $R_A$ contains $I$ if and only if the generators of $\mathscr V_Z$ vanish at $c_{\alpha\gamma}\in A$.

Hence, $\varphi(\mathscr G)$ is a $\PP{\tilde J_{\geq t}}$-marked basis in $R_A$ and the saturation of the ideal it generates in $R_A$ contains $I$ if only if $\ker(\varphi)\supseteq \mathscr U+\mathscr V_Z$. In this case, $\varphi$ factors through $\kk[C]/(\mathscr U+\mathscr V_Z)$. The induced $\kk$-algebra morphism  from $\kk[C]/(\mathscr U+\mathscr V_Z)$ to $A$ defines a scheme morphism $\mathrm{Spec}(A) \rightarrow \mathrm{Spec}(\kk[C]/(\mathscr U+\mathscr V_Z))$. Therefore, the scheme $\Spec(\kk[C]/(\mathscr U+\mathscr V_Z))$ is isomorphic to $\MFScheme{\tilde J_{\geq t}} \cap \HScheme{X}{p(z)}$.
\end{proof}

\begin{remark}
Observe that the ideal $\mathscr V_Z \subseteq \kk[C]$ depends on the chosen generating set $Z$ of $I$. However, if $Z'\subseteq \kk[C]$ is another set of polynomials generating $I$, by Yoneda's Lemma we have that  $\Spec(\kk[C]/(\mathscr U+\mathscr V_Z))\simeq \Spec(\kk[C]/(\mathscr U+\mathscr V_{Z'}))$.
\end{remark}

In the following sections we will give an alternative construction of the affine scheme $\MFScheme{\tilde J_{\geq t}} \cap \HScheme{X}{p(z)}$, for some ideals $I$.

\section{Relative marked bases and reduction relations}
\label{sec:relative marked bases}
 
Let $\tilde J\supseteq \tilde I$ be quasi-stable ideals in $R$.
With an abuse of notation, we keep on writing $\tilde J$ (resp.~$\tilde I$) for $\tilde J\cdot R_A$ (resp.~$\tilde I\cdot R_A$). %Observe that the set of terms $\PP{\tilde J}\setminus \tilde I$ coincides with $\PP{\tilde J}\setminus \PP{\tilde I}$.

\begin{definition} \label{def:rel marked set} 
A subset $H$ of a $\PP{\tilde J}$-marked set is called a {\em $\PP{\tilde J}$-marked set relative to $\tilde I$} if the head terms of the marked polynomials in $H$ are the terms in $\PP{\tilde J}\setminus \PP{\tilde I}$.
\end{definition}

Let $I$ be an ideal belonging to $\MFFunctor{\tilde I}(\kk)$, i.e.~$I$ is generated by a $\PP{\tilde I}$-marked basis $F\subseteq R$ and, equivalently, the graded decomposition $R_A=I\oplus \langle \NN(\tilde I)\rangle_{A}$ holds (with an abuse of notation, we keep on writing $I$ for $I\cdot R_A$).
 
For every polynomial $p\in R_{A}$, we denote by $\Nf_{I}(p)$ the normal form of $p$ modulo $I$, which is the unique polynomial in $\langle \mathcal N(\tilde I)\rangle_A$ such that $p-\Nf_{I}(p)\in  I$. Recall that $\Nf_{I}(p)$ is explicitely computed by the reduction relation of Definition \ref{relazione omogenea}.
%Furthermore, we denote by $[p]_I$ the equivalence class of $p$ in $R_A/I$.

\begin{lemma}\label{lemma:lemma1} With the notation above, let $J\subseteq R_A$ be an ideal containing $I$.
\begin{enumerate}[label=(\roman*)]
\item\label{it1:lemma1}
 For every polynomial $p\in J$, $\Nf_{I}(p)$ belongs to $\langle \NN(\tilde I)\rangle_{A}\cap J$. In particular, $\Nf_{I}(p)$ belongs to $J\setminus I$, unless it is null.
\item\label{it2:lemma1} The graded decomposition $J = I\oplus (\langle \NN(\tilde I)\rangle_{A} \cap J)$ holds.
\end{enumerate}
\end{lemma}

\begin{proof} The proof follows from standard arguments. 
%\begin{enumerate}
%\item[\em (i)] For every polynomial $p\in J$, it is sufficient to observe that $\Nf_{I}(p)$ belongs to $J\cap \langle \NN(\tilde I)\rangle_{A}$ because $I\subseteq J$. The second assertion holds because $\{0\}=I\cap  \langle \NN(\tilde I)\rangle_{A}$ by the hypotheses.
%\item[\em (ii)] The map $\phi: [p]_{I}\in R_{A}/ I\mapsto \Nf_{I}(p) \in  \langle \NN(\tilde I)\rangle_{A}$ is an isomorphism of $A$-modules because we have the graded decomposition $R_{A}= I\oplus \langle \NN(\tilde I)\rangle_{A} $, and so the map $[p]_{I}\in J/ I \mapsto \Nf_{I}(p) \in  \langle \NN(\tilde I)\rangle_{A} \cap J$ is an isomorphism of $A$-modules too, obtained from $\phi$ by restriction. Hence, $J \simeq I\oplus (\langle \NN(\tilde I)\rangle_{A} \cap J)$ and this isomorphism is graded and gives an equality thanks to item {\em\ref{it1:lemma1}}.\qedhere
%\end{enumerate}
\end{proof}
%
%\begin{remark}
%Both items of Lemma \ref{lemma:lemma1}  hold replacing $\tilde I$ by a homogeneous ideal $I\in \MFFunctor{\tilde I}(A)$.
%\end{remark}

\begin{definition} \label{def:rel marked basis} 
With the notation above, a $\PP{\tilde J}$-marked set $H$ relative to $\tilde I$ is called a {\em $\PP{\tilde J}$-marked basis relative to $I$} if the following graded decomposition holds, where $J\subseteq R_A$ is the ideal generated by $F\cup H$:
\begin{equation}\label{eq:decomposition}
R_A= I\oplus (\langle \NN(\tilde I)\rangle_A\cap J)% (G)
\oplus \langle \NN(\tilde J)\rangle_A.
\end{equation}
\end{definition}

\begin{theorem}\label{thm:linkToGenMethod}
With the notation above, let $H\subset R_A$ be a $\PP{\tilde J}$-marked set relative to $\tilde I$ and $J$ be the ideal generated by $F\cup H$.
Then, $H$ is a $\PP{\tilde J}$-marked basis relative to $I$ if and only if $J$ is generated by a $\PP{\tilde J}$-marked basis containing $H$.
\end{theorem}

\begin{proof}
If $J$ is generated by a $\PP{\tilde J}$-marked basis $G$  containing $H$ then $R_A=J\oplus \langle \mathcal N(\tilde J)\rangle_A$. Since $J$ contains $I$ by construction, by Lemma \ref{lemma:lemma1}\ref{it2:lemma1} we have that $J=I\oplus (\langle \NN(\tilde I)\rangle_A\cap J)$, and so $H$ is a $\PP{\tilde J}$-marked basis relative to $I$. 

If $H$ is a $\PP{\tilde J}$-marked basis relative to $I$, the decomposition \eqref{eq:decomposition} holds. 
Since $J=I\oplus  (\langle \NN(\tilde I)\rangle_{A} \cap J)$ by Lemma \ref{lemma:lemma1}\ref{it2:lemma1}, from decomposition \eqref{eq:decomposition} we obtain $R_A=J\oplus \langle \NN(\tilde J)\rangle_A$, which means that $J$ is generated by a $\PP{\tilde J}$-marked basis $G$ (see Remark \ref{rem:lemmaMP}). It remains to show that $H\subseteq G$.
%Since $H\subseteq J$, $h\crid{G^\ast}0$ for every $h\in H$ (see \cite[Corollary 5.11]{CMR2015}). 
By the hypotheses, for every $h\in H$ there is $g\in G$ with $\Ht(h)=\Ht(g)$. Hence, by construction, the polynomial $h-g$ belongs to $J\cap \langle \NN(\tilde J)\rangle_A=\{0\}$ and then $h=g$. 
\end{proof}

If $H$ is a $\PP{\tilde J}$-marked set relative to $\tilde I$, hence contained in a  $\PP{\tilde J}$-marked set $G$, we have the reduction relation $\rid{H^\ast}$ thanks to Lemma \ref{lem:subred}. 

Despite the nice properties which the reduction relation $\rid{H^\ast}$ inherits from $\rid{G^\ast}$, the following example shows that it is not always true that, for a given polynomial $p$, $p \crid{G^\ast} 0$ is equivalent either to $p \crid{H^\ast} p_{H^\ast}$ with $p_{H^\ast}\in \langle F^\ast\rangle_A$, or to $p \crid{F^\ast} p_{F^\ast}$ with $p_{F^\ast}\in \langle H^\ast\rangle_A$, for a given $\PP{\tilde I}$-marked set $F$. 

Hence, in general $(p_{H^\ast})_{F^\ast}$ and $(p_{F^\ast})_{H^\ast}$ are not the same polynomial. Moreover, the following example also shows that we might have $p=(p_{F^\ast})_{H^\ast}$ (resp. $p=(p_{H^\ast})_{F^\ast}$). This means that, although $ \rid{F^\ast}$ and $\rid{H^\ast}$ are both Noetherian, the reduction process obtained by first computing the complete reduction of a polynomial by  $\rid{F^\ast}$ (resp. by $\rid{H^\ast}$) and successively the complete reduction by $ \rid{H^\ast}$ (resp. by  $\rid{F^\ast}$) is not Noetherian.

%Hence, the combination of the relations $\rid{H^\ast}$ and $\rid{F^\ast}$ is {\em not commutative}. Moreover, although $\rid{H^\ast}$ and $\rid{F^\ast}$ are both Noetherian, their combination is {\em not Noetherian}, because as we will see  some loops can happen, even when $H$ is a relative marked basis.

\begin{example}\label{ex:counterexample}
In the polynomial ring $\kk[x_0,x_1,x_2,x_3]$, consider the quasi-stable ideal $\tilde I=(x_3^2,x_2^3)$, with Pommaret basis $\mathcal P_{\tilde I}=\{x_3^2,x_3x_2^3,x_2^3\}$, and the $\mathcal P_{\tilde I}$-marked basis $F=\{x_3^2,x_3x_2^3,x_2^3-3x_3x_2^2\}$. Let $I$ be the ideal generated by $F$. Then take the quasi-stable ideal $\tilde J=(x_3^2,x_3x_2,x_2^3)$, which contains $\tilde I$ and has  Pommaret basis $\mathcal P_{\tilde J}=\{x_3^2,x_3x_2,x_2^3\}$. Moreover, take the $\mathcal P_{\tilde J}$-marked set $H=\{x_3x_2-4x_2^2\}$ relative to $\tilde I$. In particular, {\em $H$ is a $\mathcal P_{\tilde J}$-marked basis relative to $I$} because it is contained in the $\mathcal P_{\tilde J}$-marked basis $G=\{x_3^2,x_3x_2-4x_2^2,x_2^3\}$ and $J=(G)=(H\cup F)$. 

Let $h=x_3x_2-4x_2^2$ be the unique polynomial of $H$. Then $x_3h \crid{H^\ast} f:=x_3^2x_2-16x_2^3$, where $f$ does not belong to $I=\langle F^\ast \rangle_A$. Moreover, $x_3h \crid{F^\ast} g:=-4x_2^2x_3$ where $g$ does not belong to $\langle H^\ast \rangle_A$. 

Let us now consider $x_2^3$, which is reduced with respect to $\rid{H^\ast}$. If we rewrite $x_2^3$ first by $\rid{F^\ast}$ and then by $\rid{H^\ast}$ we find $12x_2^3$, obtaining a loop, unless the characteristic of the field is $2$ or $3$.
 
If we consider $x_2^2x_3$ and rewrite it first by $\rid{H^\ast}$ and then by $\rid{F^\ast}$ we find $12x_2^2x_3$, obtaining a loop again, unless the characteristic of the field is $2$ or $3$.
\end{example}

The problems that Example \ref{ex:counterexample} highlights are new with respect to both the theory of marked bases and the theory of relative Gr\"obner bases and involutive bases. 

However, there is a relevant case in which the successive application of the relations $\rid{H^\ast}$ and $\rid{F^\ast}$ has a good behaviour. 

%\bcR insert here or somewhere that for a marked basis relative to a non-monomial ideal, we are stuck, with non-unicity problems and loops if we try to mix up reduction by $H^\ast$ and reduction by $F^\ast$ \ecr.

\begin{proposition}\label{prop:luckytails}
With the notation above, let $H$ be a $\PP{\tilde J}$-marked set relative to $\tilde I$ and $F':=\{f\in F \ \vert \ \Ht(f)\in \PP{\tilde I}\cap\PP{\tilde J}\}$. If $G:=H\cup F'$ is a $\PP{\tilde J}$-marked set, then 
\begin{itemize}
\item[(i)] For every $p\in R_A$, $p\crid{G^\ast}0$ is equivalent to $p\crid{H^\ast} r$ with $r\in \langle {F'}^\ast\rangle$.
\item[(ii)] $H$ is a $\PP{\tilde J}$-marked basis relative to $I$ if and only if $H\cup F'$ is the $\PP{\tilde J}$-marked basis of $J=(H\cup F)$.
\end{itemize}
\end{proposition}

\begin{proof}
For item (i), assume that $p$ reduces to $0$ by $G^\ast$. Since $G$ is equal to $H\cup F'$, this means that we have the following expression:
\[
p=\sum_{H^\ast}c_{\beta\eta}x^\eta h_\beta+\sum_{{F'}^\ast}c_{\alpha\gamma }x^\gamma f_\alpha.
\]
The above equality gives us that $p\crid{H^\ast}\sum_{{F'}^\ast}c_{\alpha\gamma }x^\gamma f_\alpha$, and the latter is obviously an element in $\langle {F'}^\ast\rangle_A$.
Vice versa, if $p$ reduces to an element $r\in \langle {F'}^\ast\rangle_A$ by $H^\ast$, then $r\crid{G^\ast}0$, being $F'\subset G$. Item (ii) now follows from Theorem \ref{thm:linkToGenMethod}.
\end{proof}

If $I=\tilde I$, the hypotheses of Proposition \ref{prop:luckytails} are satisfied because $F'=\PP{\tilde I}\cap\PP{\tilde J}$. 

\begin{corollary}\label{cor:ridForTildeI}
Let $\tilde J\supseteq \tilde I$ be quasi-stable ideals, $H$ a $\PP{\tilde J}$-marked set relative to $\tilde I$, and $J$ the ideal generated by $\PP{\tilde I}\cup H$. 
Then $H$ is a $\PP{\tilde J}$-marked basis relative to $\tilde I$ if and only if:
\begin{enumerate}[label=(\roman*)]
\item $\forall h_\beta \in H$, $\forall x_i>\min(\Ht(h_\beta))$, $x_i h_\beta\crid{H^\ast }r_{\beta,i}\in \tilde I$
\item $\forall x^\alpha \in \PP{\tilde I}\cap\PP{\tilde J}$, $\forall x_i>\min(x^\alpha)$, $x_ix^\alpha\crid{H^\ast }r_{\alpha,i}\in \tilde I$
\item \label{it3:thmridTildeI} $\forall x^\gamma\in \mathcal B_{\tilde I}\setminus \PP{\tilde J}$, $x^\gamma\crid{H^\ast}r_\gamma\in \tilde I$.
\end{enumerate}
\end{corollary}

\begin{proof} 
Denote by $\RR$ the set $\PP{\tilde I}\cap\PP{\tilde J}$. Observe that if $p\crid{H^\ast} r\in \tilde I$, then $r$ belongs to $\langle \RR^\ast\rangle\subseteq\tilde I$, as pointed out in Remark \ref{rem:subred}. 
Thanks to Proposition \ref{prop:luckytails}, $H$ is a $\PP{\tilde J}$-marked basis relative to $\tilde I$ if and only if $H\cup \RR$ is the $\PP{\tilde J}$-marked basis of $J=(H\cup \mathcal B_{\tilde I})$. So now it is sufficient to apply \cite[Theorem 5.13]{CMR2015} to obtain items (i) and (ii) and observe that item (iii) guarantees that $\tilde I$ is contained in $J$.
\end{proof}

\section{Relative marked functor}
\label{sec:relative marked functor}

We continue to use the same notation of Section \ref{sec:relative marked bases}, in particular $\tilde J\supseteq \tilde I$ are quasi-stable ideals in $R$ and the ideal $I\subseteq R_A$ is generated by a $\mathcal P_{\tilde I}$-marked basis $F$.

\begin{definition}
The \emph{marked functor on $\tilde J$ relative to $I$}, or, when $\tilde J$ and $I$ are well-understood, simply the \emph{relative marked functor}, is the functor
$$\MFFunctor{I, \tilde J}: \underline{\text{Noeth}\ \kk\!\!-\!\!\text{Alg}} \longrightarrow \underline{\text{Sets}}$$ 
such that
$$\MFFunctor{I, \tilde J}(A):=\{ (H\cup F) \subseteq R_A \ \vert \ H\subseteq R_A \text{ is a } \PP{\tilde J}\text{-marked basis relative to } I\}$$
and, if $\phi:A\rightarrow B$ is a morphism of Noetherian $\kk$-algebras (with $\phi(1_\kk)=1_\kk\in B)$, then the map $\MFFunctor{I, \tilde J}(\phi)$ associates to every $H\in\MFFunctor{I, \tilde J}(A)$ the $\PP{\tilde J}$-marked basis $H\otimes_A B\in \MFFunctor{I, \tilde J}(B)\subset R_B$  relative to $I$.
\end{definition}

As the reader might expect, $\MFFunctor{I, \tilde J}$ is strictly related to $\MFFunctor{\tilde J}$ and to $\HScheme{X}{p(z)}$, with $X=\Proj(R/I)$, when the hypotheses of Proposition \ref{prop:luckytails} hold.

\begin{theorem} 
\label{th:open subfunctor}
With the notation above and under the hypotheses of Proposition~\ref{prop:luckytails}, the following statements hold:
\begin{enumerate}[label=(\roman*)]
\item \label{it1:intersectionFunctors} the relative marked functor $\MFFunctor{I, \tilde J}$ is a closed subfunctor of $\MFFunctor{\tilde J}$;
\item \label{it2:intersectionFunctors} if the ideals $\tilde I$ and $\tilde J$ are both saturated, then for every integer $t\geq \rho_{\tilde J}-1$ the relative marked functor $\MFFunctor{I_{\geq t}, \tilde J_{\geq t}}$  is an open subfunctor of $\HFunctor{X}{p(z)}$ and  it is represented by $\MFScheme{\tilde J_{\geq t}} \cap \HScheme{X}{p(z)}$.
\end{enumerate}
\end{theorem} 

\begin{proof}
For what concerns item \ref{it1:intersectionFunctors}, thanks to Theorem \ref{thm:linkToGenMethod}, for every Noetherian $\kk$-algebra $A$ there is a bijection between the set $\MFFunctor{I, \tilde J}(A)$ and the set 
$$\MFFunctor{\tilde J}(A)\cap \{J\subseteq R_A : J \text{ is a homogeneous ideal containing } I\}.$$
As we already showed in Section \ref{sec:MFandHilb}, the condition \lq\lq $J$ contains $I$" can be imposed on the ideals in $\MFFunctor{\tilde J}(A)$ by further closed conditions on the polynomials generating the ideal that defines the scheme which represents the functor $\MFFunctor{\tilde J}$. So, we obtain item \ref{it1:intersectionFunctors}.

Thanks to \cite[Theorem 3.5 and Corollary 3.7]{BCRAffine}, item \ref{it1:intersectionFunctors}, and Lemma \ref{lem:satEquiv}(ii), for every integer $t$, $\MFFunctor{I_{\geq t}, \tilde J_{\geq t}}$ is a closed subfunctor of $\MFFunctor{\tilde J_{\geq t}}$. Moreover, if $t\geq \rho_{\tilde J}-1$, we have $\MFFunctor{I_{\geq t}, \tilde J_{\geq t}}(A)=\MFFunctor{\tilde J_{\geq t}}(A) \cap \HFunctor{X}{p(z)}(A)$, and item~\ref{it2:intersectionFunctors} holds thanks to Proposition \ref{prop:sottoschema}.
\end{proof}

\subsection{Case $I=\tilde I$.} In the particular case $I=\tilde I$, we now give a construction of the scheme representing $\MFFunctor{\tilde I, \tilde J}$ which is alternative to the construction that is described in Section \ref{sec:MFandHilb} for $\MFScheme{\tilde J_{\geq t}} \cap \HScheme{X}{p(z)}$. We use the computational method that arises from Corollary~\ref{cor:ridForTildeI} in order to  characterize  relative marked bases.

Let $C'$ denote the finite set of variables
$$\bigl\{C_{\beta\eta}\mid x^\beta \in \PP{\tilde J}\setminus \PP{\tilde I}, x^\eta \in \NN(\tilde J ), \deg(x^\eta)=\deg(x^\beta)\bigr\}$$ 
and consider the $\kk$-algebra $\kk[C']$. Then, we construct the set $\mathscr  H\subset R_{\kk[C']}$ consisting of the following marked polynomials
\begin{equation}\label{eq:polinomi h}
h_\beta=x^\beta-\sum_{x^{\eta}\in \NN(\tilde J)_{\vert\beta\vert}}C_{\beta\eta }x^\eta
\end{equation}
with $x^{\beta}\in\PP{\tilde J}\setminus \mathcal P_{\tilde I}$.
Moreover, we consider 
$$\mathscr H^\ast :=\{x^\delta h_\beta \ \vert \ h_\beta \in \mathscr H, x^\delta \in \mult(h_\beta)\}.$$
We highlight that the set $C'$ can be identified to a subset of the set $C$ given in Section \ref{sec:MFandHilb}, and up to this identification we can consider $\mathscr H$ as a subset of  $\mathscr G$.

Then, we explicitly compute the following polynomials in  $R_{\kk[C']}$ by $\rid{\mathscr H^\ast}$:
\begin{itemize}
\item $\forall h_\beta \in \mathscr H$, $\forall x_i>\min(\Ht(h_\beta))$, let $r_{\beta,i}$ be such that $x_i h_\beta\crid{\mathscr H^\ast }r_{\beta,i}$;
\item  $\forall x^\alpha \in \RR$, $\forall x_i>\min(x^\alpha)$, let $r_{\alpha,i}$ be such that $x_ix^\alpha\crid{\mathscr H^\ast }r_{\alpha,i}$;
\item $\forall x^\gamma\in \mathcal B_{\tilde I}\setminus \PP{\tilde J}$, let $r_{\gamma}$ be such that $f\crid{\mathscr H^\ast}r_\gamma$.
\end{itemize}
For every $ h_\beta \in \mathscr H$, and for every $x_i>\min(\Ht(h_\beta))$, we collect the coefficients in $\kk[C']$ of the terms in $\supp(r_{\beta,i})$ not belonging to $\tilde I$, and the same for all the polynomials $r_{\alpha,i}$ and $r_\gamma$. Let $\mathscr R\subset \kk[C']$ be the ideal generated by these coefficients.
 
\begin{theorem}\label{th:representationTildeI}
The functor $\MFFunctor{\tilde I, \tilde J}$ is the functor of points of $\Spec(\kk[C']/\mathscr R)$, which we denote by $\MFScheme{\tilde I, \tilde J}$.
\end{theorem}

\begin{proof}
Let $H\subseteq R_A$ be a $\mathcal P_{\tilde J}$-marked basis relative to $\tilde I$ and denote by $\phi_H$ the evaluation morphism $\phi_H: \kk[C'] \to A$ that associates to every variable in $C'$ the corresponding coefficient in the polynomials of $H$. 
It is sufficient to observe that $H$ is a $\mathcal P_{\tilde J}$-marked basis relative to $\tilde I$ if and only if $\phi_H$ factors through $\kk[C']/\mathscr R$, or in other words if and only if the following diagram commutes
\[
\xymatrix{
\kk[C']  \ar[rr]^{\phi_H} \ar[dr]& &A\\
& \kk[C']/\mathscr R\ar[ru]}\,.
\]
Equivalently, $H$ is a $\mathcal P_{\tilde J}$-marked basis relative to $\tilde I$ if and only if $\mathscr R$ is contained in $\mathrm{ker}(\phi_H)$, which is true thanks to Corollary~\ref{cor:ridForTildeI}.
\end{proof}

The scheme $\MFScheme{\tilde I, \tilde J}$ which has been introduced in the statement of Theorem \ref{th:representationTildeI} is called a {\em relative marked scheme}. 

\begin{remark}\label{rem:number parameters}
The scheme $\Spec(\kk[C']/\mathscr R)$ is computationally more advantageous compared with $\Spec(\kk[C]/(\mathscr U+\mathscr V_F))$ considered in Theorem \ref{thm:constrRelScheme1}. Indeed, if $\PP{\tilde I}\cap \PP{\tilde J}\neq \emptyset$, then $\vert C'\vert <\vert C\vert$ and the reduction $\rid{\mathscr H^\ast}$ involves the relative marked set $\mathscr H$, which contains less polynomials than  $\mathscr G$. Actually, in principle we perform less reduction steps using $\rid{\mathscr H^\ast}$.
%, which is a subreduction of $\rid{\mathscr G^\ast}$.
%However, observe that this procedure does not work for a non-monomial ideal $I$, because Lemma \ref{lem:equivRed} (and hence Theorem \ref{cor:ridForTildeI}) do not hold.
\end{remark}

If the ideals $\tilde I$ and $\tilde J$ are both saturated and we consider $\MFFunctor{\tilde I_{\geq t}, \tilde J_{\geq t}}$ for some $t$, we give a further different presentation of the scheme representing it. Take the following polynomials in $\kk[C']$:
\begin{equation}\label{eq:fewForSat}
\forall x^\gamma\in \PP{\tilde I}\setminus \PP{\tilde J}, \text{ let $r_{\gamma}$ be such that }\\  x_0^{\max\{0, t-\deg(x^\gamma)\}}x^\gamma \crid{\mathscr H^\ast} r_\gamma\tag{$\star$}
\end{equation}
Observe that if $t$ is strictly bigger than the initial degree of $\tilde I$, then $\vert\PP{\tilde I}\setminus \PP{\tilde J}\vert$ is strictly smaller than $\vert \PP{\tilde I_{\geq t}}\cap \PP{ \tilde J_{\geq t}}\vert$. 
 
Let $\mathscr R'\subset \kk[C']$ be the ideal generated by the coefficients in $\kk[C']$ of the terms not belonging to $\tilde I$ of the polynomials $r_{\beta,i}$, $r_{\alpha,i}$ considered for $\mathscr R$, and by the coefficients in $\kk[C']$ of the terms not belonging to $\tilde I$ of the polynomials $r_\gamma$ in \eqref{eq:fewForSat}.
 
\begin{theorem}
The functor  $\MFFunctor{\tilde I_{\geq t}, \tilde J_{\geq t}}$ is the functor of points of $\Spec(\kk[C']/\mathscr R')$.
\end{theorem}

\begin{proof}
The proof is analogous to that of Theorem \ref{th:representationTildeI} thanks to Lemma \ref{lem:satEquiv}.
\end{proof}

Algorithm \ref{algorithm} collects the instructions to compute the ideal $\mathscr R'$.

\begin{algorithm}[ht]
\caption{\label{algorithm} Algorithm for computing the defining ideal $\mathscr R'$ representing the relative marked functor $\MFFunctor{\tilde I_{\geq t}, \tilde J_{\geq t}}$}
\begin{algorithmic}[1]
\STATE MarkedFunctor($\tilde J,\tilde I,t$)
\REQUIRE saturated quasi-stable ideals $\tilde J\supseteq \tilde I$ and a non-negative integer $t$
\ENSURE generators of the ideal $\mathscr R'$ representing the relative functor $\MFFunctor{\tilde I_{\geq t}, \tilde J_{\geq t}}$
\STATE let $\mathscr H\subseteq \kk[C']$ be the set of the polynomials defined in \eqref{eq:polinomi h} with respect to the quasi-stable ideals $\tilde J_{\geq t}$ and $\tilde I_{\geq t}$
\STATE $\mathscr R':=(0)$
\FOR{$h_\beta \in \mathscr H$}
\FOR{$x_i>\min(\mathrm{Ht}(h_\beta))$}
\STATE compute $r_{\beta,i}$ such that $x_i h_\beta \crid{\mathscr H^\ast } r_{\beta,i}$
\STATE $\mathscr R':=\mathscr R'+(\text{coefficients in $r_{\beta,i}$ of the terms not belonging to } \tilde I)$
\ENDFOR
\ENDFOR
\FOR{$x^\alpha \in \mathcal P_{\tilde J_{\geq t}}\cap \mathcal P_{\tilde I_{\geq t}}$}
\FOR{$x_i>\min(x^\alpha)$}
\STATE compute $r_{\alpha,i}$ such that $x_i x^\alpha \crid{\mathscr H^\ast } r_{\alpha,i}$
\STATE $\mathscr R':=\mathscr R'+(\text{coefficients in $r_{\alpha,i}$ of the terms not belonging to } \tilde I)$
\ENDFOR
\ENDFOR
\FOR{$x^\gamma \in \mathcal B_{\tilde I}\setminus \mathcal P_{\tilde J}$}
\STATE compute $r_{\gamma}$ such that $x_0^{\max\{0,t-\deg(x^\gamma)\}}x^\gamma \crid{\mathscr H^\ast } r_{\gamma}$
\STATE $\mathscr R':=\mathscr R'+(\text{coefficients in $r_{\gamma}$ of the terms not belonging to } \tilde I)$
\ENDFOR
\STATE return $\mathscr R'$
\end{algorithmic}
\end{algorithm}

In the following examples we take into account the fact that, thanks to Theorem~\ref{th:open subfunctor}\ref{it2:intersectionFunctors}, the relative marked scheme $\MFScheme{\tilde I_{\geq t}, \tilde J_{\geq t}}$ defined by the ideal $\mathscr R'$ that has just introduced is an open subscheme of $\HScheme{X}{p(z)}$. Hence, the Zariski tangent space to this open subscheme at one of its points is equal to the Zariski tangent space to $\HScheme{X}{p(z)}$ at the same point (see also \cite[Corollary 1.9]{Gore}). 

\begin{example}\label{ex:first example}
This is an example of scheme $\MFScheme{\tilde I_{\geq t}, \tilde J_{\geq t}}$ which is neither irreducible nor reduced.
Take the quasi-stable ideals $\tilde I=(x_3^2,x_2^5)\subset \tilde J=(x_3^2,x_3x_2,x_3x_1^2,x_2^5) \subset R:=\kk[x_0,\dots,x_3]$, with $\mathcal P_{\tilde J}=\{x_3^2,x_3x_2,x_3x_1^2,x_2^5\}$ and $\mathcal P_{\tilde I}=\{x_3^2,x_3x_2^5,x_2^5\}$.
Following Algorithm \ref{algorithm} and using CoCoA  \cite{CoCoA}, we compute the ideal $\mathscr R'$ defining the relative marked scheme $\MFScheme{\tilde I_{\geq t}, \tilde J_{\geq t}}$ for $t=\rho_{\tilde J}-1=2$. In this case we have $\tilde J_{\geq t}=\tilde J$ and $\tilde I_{\geq t}=\tilde I$. 
The set $\mathscr H$ is made of the following polynomials in the ring $\mathbb Q[c_1,\dots,c_{20}][x_0,\dots,x_3]$:

$h_1= c_1x_0^2 + c_2x_0x_1 + c_3x_1^2 + c_4x_0x_2 + c_5x_1x_2 + c_6x_2^2 + c_7x_0x_3 + c_8x_1x_3 + x_2x_3,$

$h_2= c_9x_0^3 + c_{10}x_0^2x_1 + c_{11}x_0x_1^2 + c_{12}x_1^3 + c_{13}x_0^2x_2 + c_{14}x_0x_1x_2 + c_{15}x_1^2x_2 + c_{16}x_0x_2^2 +$ 

\hskip 9.5mm $c_{17}x_1x_2^2 + c_{18}x_2^3 + c_{19}x_0^2x_3 + c_{20}x_0x_1x_3 + x_1^2x_3$.

\vskip 1mm
\noindent By $\rid{\mathscr H^\ast }$ we reduce the polynomials $x_3 h_1$, $x_3 h_2$, $x_2 h_2$, $x_3x_2^5$ and then apply the reduction modulo $\tilde I$, obtaining the ideal $\mathscr R' \subseteq \mathbb Q[c_1,\dots,c_{20}]$.
The ring $\mathbb Q[c_1,\dots,c_{20}]/\mathscr R'$ has Krull dimension $2$, so $\MFScheme{\tilde I_{\geq t}, \tilde J_{\geq t}}$ has dimension $2$. Moreover, the Zariski tangent space to $\MFScheme{\tilde I_{\geq t}, \tilde J_{\geq t}}$ at $Y$ has dimension~$7$ and the point $Y$ defined by $\tilde J$ is singular in $\HScheme{X}{p(z)}$, where $p(z)=5z-3$ is the Hilbert polynomial of $R/\tilde J$ and $X=\mathrm{Proj}(R/\tilde I)$.

Using Macaulay 2  \cite{M2}  we compute the irreducible components of $\mathscr R'$, obtaining that the associated primes have both dimension $2$ and are
\[\mathscr P_0=\left(c_{18},\,c_{17},\,c_{16},\,c_{15},\,c_{14},\,c_{13},\,c_{12},\,c_{11},\,c_{10},\,c_{9},\,c_{8},\,c_{7},\,c_{6},\,c_{5},\,c_{4},\,c_{3},\,c_{2},\,c_{1}\right),\]
\begin{equation*}
\begin{split}
\mathscr P_1=(&c_{18},\,c_{17},\,c_{16},\,c_{15},\,c_{14},\,c_{13},\,c_{12},\,c_{11},\,c_{10},\,c_{9},\,c_{6},\,c_{5},\,c_{4},\,c_{3},\,c_{2},\,c_{1}\,c_{20}^{2}-4c_{19},\\
&c_{8}c_{20}-2\,c_{7},\,2\,c_{8}c_{19}-c_{7}c_{20}).
\end{split}
\end{equation*}
The defining ideals of both the two irreducible components are not prime and contain the point $Y$ as a singular point, because the Zariski tangent spaces at $Y$ to these components have dimensions $7$, too. Some ancillary material related to this example is available at \url{http://www.dma.unina.it/~cioffi/RedirectRelative.html}.

If we compute the generators for the ideal $\mathscr U+\mathscr V_Z$ of Theorem~\ref{thm:constrRelScheme1}, then we have to deal with $50$ parameters instead of $20$.
\end{example}

\begin{example}\label{ex:second example} 
This is an example of scheme $\MFScheme{\tilde I_{\geq t}, \tilde J_{\geq t}}$ which is irreducible but not reduced. Assume $\mathrm{char}(\kk)\neq 2$. Given an integer $p> 2$, take the quasi-stable ideals $\tilde I=(x_n^2,x_{n-1}^p)\subset \tilde J=(x_n,x_{n-1}^p)\subseteq R:=\kk[x_0,\dots,x_n]$, with $\mathcal P_{\tilde J}=\{x_n,x_{n-1}^p\}$ and $\mathcal P_{\tilde I}=\{x_n^2,x_nx_{n-1}^p,x_{n-1}^p\}$. The ideal $\tilde J$ defines the only point $Y$ of the Hilbert scheme $\HScheme{X}{p(z)}$, where $p(z)$ is the Hilbert polynomial of $R/\tilde J$ and $X=\mathrm{Proj}(R/\tilde I)$ (see \cite{G}, for instance).
Following Algorithm \ref{algorithm}, by hand we compute the ideal $\mathscr R'$ defining $\MFScheme{\tilde I_{\geq t}, \tilde J_{\geq t}}$ for $t=\rho_{\tilde J}-1=0$, hence $\tilde J_{\geq t}=\tilde J$ and $\tilde I_{\geq t}=\tilde I$. 
The set $\mathscr H$ is made only of the polynomial 
$h=c_1x_0 + c_2x_1 + c_3x_2 + \dots +c_{n}x_{n-1} + x_n$. 
By $\rid{\mathscr H^\ast }$ we reduce the terms $x_nx_{n-1}^p$ and $x_n^2$ and then apply the reduction modulo $\tilde I$, obtaining the ideal $\mathscr R'$:
$$\mathscr R'=(c_n,\dots,c_2,c_1)^2.$$
The affine scheme $\mathrm{Spec}(\kk[c_1,c_2,\dots,c_n]/\mathscr R')$ is a zero-dimensional scheme supported over the origin and with Zariski tangent space of dimension $n$ at $Y$. By definition, the multiplicity in $\HScheme{X}{p(z)}$ of the fat point $Y$ is $n+1$.

If we compute the generators for the ideal $\mathscr U+\mathscr V_Z$ of Theorem~\ref{thm:constrRelScheme1}, then we have to deal with $\binom{n-1+p}{p}+n-1$ parameters instead of $n$. %In particular, $\MFScheme{\tilde J_{\geq t}}$ is a rational open subscheme in $\HScheme{\mathbb P^n_{\kk}}{p(z)}$ of dimension $23$. CONTROLLARE QUI
\end{example}

\section{An open cover of a Hilbert scheme over a Cohen-Macaulay  quotient ring on a quasi-stable ideal}
\label{sec:open cover}

Let $\tilde I$ be a saturated quasi-stable ideal, $S:=R/\tilde I$, and consider  the projective scheme $X=\mathrm{Proj}(\QR)$. Hence we can consider the Hilbert scheme $\HScheme{X}{p(z)}$, for an admissible Hilbert polynomial $p(z)$.

When we consider the image in $\QR$ of an element $f$ of $R$ we mean its image $[f]$ by the projection $\pi: R \to R/\tilde I$, that is its residue class modulo $\tilde I$. 

Following \cite{MP}, a term $\tau$ of $R$ is said {\em $\tilde I$-free} if its residue class $[\tau]$ is non-null, i.e.~$\tau$ does not belong to $\tilde I$. A term of $\QR$ is the image in $\QR$ of an $\tilde I$-free term of $R$. If $W$ is any set of terms in $\QR$, by abuse of notation we will use the symbol $W$ to also denote the set of terms in $\pi^{-1}(W)$ and vice versa.

The ring $\QR$ inherits from $R$ a grading for which  a term $[t]\in \QR$ has degree $q\geq 0$ if and only if $t$ has the degree $q$ in $R$.

An ideal $U\subseteq \QR$ is {\em monomial} if it is the image in $\QR$ of a monomial ideal of $R$. Every monomial ideal $U$ of $\QR$ has a unique minimal generating set $\mathcal B_U$ made of terms of~$\QR$. A monomial ideal $U$ of $\QR$ will be said {\em quasi-stable in $\QR$} if the ideal $(\mathcal B_{U}\cup \mathcal B_{\tilde{I}})$ is quasi-stable in $R$, that is $U$ is the image of a quasi-stable ideal of $R$.

\begin{definition}\label{def:U-marked set}
Let $U$ be a quasi-stable ideal in $\QR$ and let $\tilde J=(\mathcal B_U\cup \mathcal B_{\tilde I})$. The image in $\QR$ of a $\mathcal P_{\tilde J}$-marked set relative to $\tilde I$ is called a {\em $U$-marked set}. The image of a $\mathcal P_{\tilde J}$-marked basis relative to $\tilde I$ is called a {\em $U$-marked basis}.
\end{definition}

%Since we want to consider quotients $\QR=R/\tilde I$ over quasi-stable ideals $\tilde I$ that are even Cohen-Macaulay, we now recall a Cohen-Macaulay characterization for quasi-stable ideals.

%For any quasi-stable ideal, there is an index $0\leq i\leq n$ such that for each $i\leq j\leq n$, a pure power of the form $x_j^{a_j}$ is contained in the ideal. Moreover, for every index $\ell<i$ there is no pure variable power $x_\ell^{b_\ell}$ in the ideal. This minimal index of a variable appearing in a pure power can be used to characterize Cohen-Macaulay quasi-stable ideals.

From now, we also assume that the ring $\QR:=R/\tilde I$ is Cohen-Macaulay. 

\begin{proposition}[{\cite[Theorem 3.20]{Seiler2009II}, \cite[Theorem 5.2.9]{Seiler:InvolutionBook}}]\label{prop:CohenMacaulayQuasiStableIdeals}
Let $\tilde I$ be a quasi-stable ideal in $R$ and let $\mathcal{P}_{\tilde I}$ be its Pommaret basis. Then $\QR=R/ \tilde I$ is Cohen-Macaulay if and only if, for the integer $m=\min\{\min(x^{\alpha})\mid x^\alpha \in \mathcal{P}_{\tilde I} \}$, there is a pure variable power $x_m^{a_m}$ in $\tilde{I}$.
\end{proposition}

\begin{remark}\label{rem:StableCohenMacaulay}
Note that the criterion given in Proposition \ref{prop:CohenMacaulayQuasiStableIdeals} can be applied to a  quasi-stable ideal $\tilde I$ by looking at the minimal value of $\min(x^\alpha)$ for  $x^\alpha$ in~$\mathcal B_{\tilde I}$.
\end{remark}

Let $x_{k+1},\dots,x_n$ be the variables that divide some minimal generators of $\tilde I$.
In the following discussion, we will write $\T^{\prime}$ for the set of all terms in the polynomial ring $\kk[x_{k+1},\ldots,x_n]$ and $\T^{\prime\prime}$ for the set of all terms in the polynomial ring $\kk[x_0,\ldots,x_k]$.

We need to adapt \cite[Proposition 7.2 and Corollary~7.4]{Quot} to our current setting. We make the following observation as a first step.

\begin{lemma}\label{lem:FiniteModuleDecomp}
The Cohen-Macaulay quotient ring $\QR=R/\tilde{I}$ is a finitely generated graded free $\kk[x_0,\ldots,x_k]$-module
$$\QR=\bigoplus_{e=0}^d \left(\kk[x_0,\ldots,x_k](-e)\right)^{m_e},$$
where $d=\max\{\deg(t)\mid t\in\T^{\prime}\setminus \tilde{I}\}$ and, for each $0\leq e\leq d$, $m_e=|\{t\in \T^{\prime}\setminus \tilde{I}\mid \deg(t)=e\}|$.
\end{lemma}

\begin{proof}
First, note that $\tilde{I}$ contains pure powers of all variables $x_j$ with $j>k$; hence, the set $\mathcal{T}(e):=\{t\in \T^{\prime}\setminus \tilde{I}\mid \deg(t)=e\}$ is finite. Since the generators of $\tilde{I}$ are terms in $\mathbb{T}^\prime$, we have for each $t\in \mathcal{T}(e)$ an injection 
$$\iota: \mathbb{T}^{\prime\prime}\rightarrow \QR, u\mapsto [u\cdot t].$$ 

Now, we turn to the graded decomposition. Recall that the ring $\QR$ inherits a grading from $R$. It is easy to see that the set of terms of degree $q\geq 0$ in $\QR$ is disjointly decomposed as follows:
\begin{equation}\label{eqn:FiniteFreeModuleDecomposition}
    S_q\cap \left\{[u]\mid u\in\T\right\}=\bigsqcup_{e=0}^d \bigsqcup_{t\in \mathcal{T}(e)} t\cdot \left(S_{q-e}\cap \left\{[u]\mid u\in\T^{\prime\prime}\right\}\right).
\end{equation}
It is important to note that all elements in the sets of the right hand side of \eqref{eqn:FiniteFreeModuleDecomposition} are non-zero; this is guaranteed by the Cohen-Macaulay property of $\tilde{I}$. The claim follows.
\end{proof}

We now highlight that the definition of quasi-stable ideal $U$ of $\QR$ as the image in $\QR$ of a quasi-stable ideal $\tilde J$ in $R$ containing $\tilde I$ is equivalent to the definition of quasi-stable submodule of a free module given in \cite[Definition~3.2 item (i)]{Quot}.

\begin{proposition}\label{prop:EquivalenceQuasiStableCharacterizations}
Let $U\subseteq \QR$ be a monomial ideal with minimal generating set $\mathcal B_U$. Then $U$ is quasi-stable if and only if $\mathcal B_U$, interpreted as a monomial subset of the $\kk[x_0,\ldots,x_k]$-module $\QR$, generates a quasi-stable submodule.
\end{proposition}

\begin{proof}
The proof is by routine verification of quasi-stability conditions (as given in \cite[Definition 2.2]{Quot}) for the minimal generators of the ideals and submodules that are considered.
\end{proof}

Thanks to Proposition \ref{prop:EquivalenceQuasiStableCharacterizations}, a quasi-stable ideal $U$ of $\QR$ can be even considered as a quasi-stable $\kk[x_0,\ldots,x_k]$-submodule of $\QR$, and vice versa.

However, the notions of marked set over a quasi-stable ideal $U$ introduced in Definition \ref{def:U-marked set} and of marked set over a submodule given in \cite[Definition 4.3]{Quot} are different, as we will see in Example \ref{ex:important}. 

\begin{example}\label{ex:important}
Take $R=\kk[x_0,x_1,x_2]$, $\tilde{I}=(x_2^7)$ and $\tilde{J}=(x_1^2,x_2^7)$. Let $U$ be the ideal that is the image of $\tilde J$ in $\QR$. Observe that $\mathcal{P}_{\tilde{J}}=\{x_1^2,x_1^2x_2,\ldots,x_1^2x_2^6,x_2^7\}$. Thus,  according to Definition \ref{def:U-marked set}, the set $\hat{F}=\{x_0x_1+x_1^2\}$ is not a $U$-marked set (for such a set, we would need additionally polynomials marked on each of the terms $x_1^2x_2,\ldots,x_1^2x_2^6$). Nevertheless, $\hat{F}$ is marked on the Pommaret basis of the quasi-stable monomial $\kk[x_0,x_1]$-submodule of $\QR$ generated by $\{x_1^2\cdot[1]\}$, because this singleton set is also the Pommaret basis of the submodule generated by it (see \cite[Definition~ 3.1]{Quot}).
\end{example}

Nevertheless, for high degrees $q$, the notions of marked sets over a quasi-stable submodule and of $U$-marked sets for a quasi-stable ideal $U\subseteq \QR$ coincide.

Recall that the regularity $\mathrm{reg}(\tilde I)$ of a quasi-stable ideal $\tilde I$ coincides with the ma\-xi\-mum degree of a term in its Pommaret basis $\mathcal P_{\tilde I}$. Hence, for every $q\geq \mathrm{reg}(\tilde I)$, the Pommaret basis of $U_{\geq q}$ coincides with $\mathcal B_{U_{\geq q}}$.

\begin{corollary}\label{cor:solution}
Let $U\subset \QR$ be a quasi-stable ideal and $\tilde J$ be a quasi-stable ideal whose image in $\QR$ is $U$. Let $F\subseteq S_q$ be a finite set of homogeneous elements of $\QR$ of degree $q\geq \max\{\mathrm{reg}(\tilde J), \mathrm{reg}(\tilde I)\}$. Then, $F$ is a $U_{\geq q}$-marked basis if and only if it is a marked basis over the submodule $U_{\geq q}$.
\end{corollary}

\begin{proof}
We already observed that in the present setting a quasi-stable monomial ideal $U$ of $\QR$ can be even considered as a quasi-stable $\kk[x_0,\ldots,x_k]$-submodule of $\QR$, and vice versa.
Indeed, by Proposition \ref{prop:EquivalenceQuasiStableCharacterizations}, in both the two above interpretations $\mathcal B_U$ is a set of generators of $U$.

The difference between Definition \ref{def:U-marked set} and \cite[Definition 4.3]{Quot} appears when $\mathcal B_{\tilde J}$ does not coincide with $\PP{\tilde J}$, so that $\mathcal B_U$ does not contain all the terms on which we expect marked polynomials in a $U$-marked basis. 

If $q\geq \max\{\mathrm{reg}(\tilde J), \mathrm{reg}(\tilde I)\}$, then $\mathcal B_{\tilde J_{\geq q}}=\PP{\tilde J_{\geq q}}$ and $\mathcal B_{\tilde I_{\geq q}}=\PP{\tilde I_{\geq q}}$. Hence $\mathcal B_{U_{\geq q}}=\PP{\tilde J_{\geq q}}\setminus \PP{\tilde I_{\geq q}}=\mathcal B_{\tilde J_{\geq q}}\setminus \mathcal B_{\tilde I_{\geq q}}$ contains all the expected terms for a $U_{\geq q}$-marked basis.
\end{proof}

\begin{example}\label{ex:TransformedMkdSetInHighDegrees}
Consider $\tilde{I}=(x_2^7)\subseteq R=\kk[x_0,x_1,x_2]$, and $\hat{F}=\{x_0x_1^6+x_1^7\}$, which is marked on $\{x_1^7\}$, and $U=(x_1^7)$ is a quasi-stable ideal in $\QR$. $\hat{F}$ is not a $U$-marked set in the sense of Definition \ref{def:U-marked set}, because the Pommaret basis $\mathcal{P}_{\tilde{J}}$ of $\tilde{J}=(x_1^7,x_2^7)\subseteq R$ includes also the terms $x_1^7x_2^a$ for $1\leq a\leq 6$. Note that the degrevlex leading ideal of $J=(\hat{F},\tilde{I})$ is exactly $\tilde{J}$; this implies that $J$ (and hence also $(F,\tilde{I})$) is $13$-regular ($13$ being the highest degree of an element of $\mathcal{P}_{\tilde{J}}$).

Now consider, as above, $\tilde{I}=(x_2^7)$, but set $F=\{x_1x_2^6\}$. We have $F\subset S_7$, and $7=\mathrm{reg}(\tilde{I})$; moreover $F$ is already marked on $\{x_1x_2^6\}$ which generates a  quasi-stable ideal $U=(x_1x_2^6)\subseteq \QR$. Since the Pommaret basis $\mathcal{P}_{\tilde{J}}$ of $\tilde{J}=(F,\tilde{I})$ is exactly $F\cup \{ x_2^7 \}$, $F$ is a $U$-marked set in the sense of Definition \ref{def:U-marked set}. Note that the ideal $J=\tilde{J}=(F,\tilde{I})\subseteq R$ is $7$-regular, because it is a quasi-stable monomial ideal whose minimal Pommaret basis has maximal degree $7$.
\end{example}

In the remaining part of this section, we assume that the field $\kk$ is infinite. 

We denote by $\PGL{k+1}$ the subset of $\mathrm{PGL}_{\kk}(n+1)$ whose elements define invertible change of coordinates of the following kind:
$$x_i\mapsto x_i \quad\text{for}\;\,i=k+1,\dots,n,\quad x_j\mapsto \sum_{t=0}^k g_{jt}x_t\quad \text{for}\;\,j=0,\dots,k.$$
For any element $g\in \PGL{k+1}$ we denote by $\tilde g$ the automorphism induced by $g$ on $\QR$.

\begin{proposition}\label{prop:QuasiStableTrafo}
For a given degree $q\geq 0$, let $F\subset S_q$ be a finite set.
 There exists $g\in \PGL{k+1}$ such that $\tilde{g}(F)$ becomes after an autoreduction a marked set over a quasi-stable monomial submodule of $\QR$ (according to \cite[Definition~4.3]{Quot}).
\end{proposition}

\begin{proof}
We can directly apply \cite[Corollary~7.4]{Quot}, because $F$ is a subset of a single degree component of the finitely generated free graded $\kk[x_0,\ldots,x_k]$-module $\QR$ by Lemma \ref{lem:FiniteModuleDecomp}.
\end{proof}

%\begin{remark}
%The set $\hat F$ of Examples \ref{ex:important} is obtained applying to $x_0x_1$ the transformation $x_0 \mapsto x_0+x_1$. 
%\end{remark}

%Proposition \ref{prop:QuasiStableTrafo} guarantees that any finite set of homogeneous elements in $\QR$ of the same {\em big enough} degree $q$ can be transformed into a $U$-marked set, for a suitable quasi-stable ideal $U$. The following result is a consequence.

\begin{corollary}\label{cor:tranformation}
For every field extension $\mathbb L$ of $\kk$, let $J\subseteq R_{\mathbb L}$ be a saturated ideal containing $\tilde I$ and $t$ be an integer such that $t\geq \max\{\mathrm{reg}(J),\mathrm{reg}(\tilde I)\}$. Then, there exists $g\in \PGL{k+1}$ such that the ideal $\tilde g(J_t) \cdot \QR$ is generated by the image in $\QR$ of a $\mathcal P_{\tilde J}$-marked basis $H$ relative to $\tilde I$ that belongs to $\MFFunctor{\tilde I_{\geq t}, \tilde J_{\geq t}}(\mathbb L)$, for some saturated quasi-stable ideal $\tilde J$ containing $\tilde I$.
\end{corollary}

\begin{proof}
%First we observe that $(\tilde I_t)$ is contained in $(J_t)$.

By the assumptions on $t$, we can take a set $F$ of generators of $(J_t)/\tilde I$ made only of elements of degree $t$, so that $F\subseteq S_t$. Recall that we are now assuming that $\kk$ is infinite, and hence Zariski dense in any field extension $\mathbb L$. 

Then by Proposition \ref{prop:QuasiStableTrafo}, there exists $g\in \PGL{k+1}$ such that $\tilde{g}(F)$ yields after an autoreduction a marked set $H$ over a quasi-stable monomial submodule of $\QR$. By \cite[Lemma~7.5]{Quot}, we can assume that $H$ is a marked basis over the stable module $U=(\mathrm{Ht}(H))$ (see \cite[Definition~3.2]{Quot}) with regularity $\leq t$. Hence, the ideal $(\mathrm{Ht}(H),\mathcal{P}_{\tilde{I}_{\geq t}})$ has regularity $\leq t$, and its saturation $\tilde J$ contains $\tilde I$ by Lemma \ref{lem:satEquiv}. Thus, thanks to  Corollary \ref{cor:solution}, $H$ is also a $U$-marked basis and we obtain the thesis by Definition \ref{def:U-marked set}.
\end{proof}

Recall that if $J$ belongs to $\MFFunctor{\tilde J}(\kk)$ then $J_{\geq t}$ belongs to $\MFFunctor{\tilde J_{\geq t}}(\kk)$, but the converse is not true (e.g. \cite[Example 3.8]{BCRAffine}). However, by Lemma \ref{lem:satEquiv}, if $J_{\geq t}$ belongs to $\MFFunctor{\tilde J_{\geq t}}(\kk)$ and contains $\tilde I_{\geq t}$, then $J$ contains $\tilde I$.

%Let $t$ be an integer higher than or equal to the maximum between the Gotzmann number of a given Hilbert polynomial $p(z)$ and the regularity of $\tilde I$. 

Given the quasi-stable ideal $\tilde I$, let $p(z)$ be any Hilbert polynomial as in Section~\ref{sec:MFandHilb}, and consider the sets
$$Q_{p(z)}:=\{\tilde J \text{ saturated quasi-stable} \ \vert \  \QR/\tilde J \text{ has Hilbert polynomial } p(z) \},$$
$$Q_{p(z),\tilde I}:=\{\tilde J \text{ saturated quasi-stable} \ \vert \ \tilde J\supseteq \tilde I \text{ and } \QR/\tilde J \text{ has Hilbert polynomial } p(z) \}.$$
%Moreover, let $X=\mathrm{Proj}(\QR)$ and denote by $\mathrm{PGL}_{p(z),\tilde I}$ the subset of $\mathrm{PGL}$ made of the transformations $g$ that satisfy Corollary \ref{cor:tranformation} for some saturated ideal $J$ defining a point of $\HScheme{X}{p(z)}$. It is noteworthy that every $g\in \mathrm{PGL}_{p(z),\tilde I}$ fixes the ideal $\tilde I$.
%for which $\tilde g(J_{\geq t})\cdot \QR$ is determined by a relative basis in $\MFFunctor{\tilde I_{\geq t}, \tilde J_{\geq t}}(\mathbb K)$, for some ideal $\tilde J \in Q_{p(z),\tilde I}$. 
The \emph{Gotzmann number} $r$ of the Hilbert polynomial $p(z)$ is the smallest integer such that $r\geq \mathrm{reg}(J)$ for every  saturated ideal $J$ defining a scheme lying on $\HScheme{\mathbb P^n}{p(z)}$. 

For every $g \in \PGL{k+1}$, we consider the functor $\MFFunctor{\tilde I_{\geq t}, \tilde J_{\geq t}}^{\tilde g}$  that assigns to every $\kk$-algebra $A$ the set $\{(\tilde g^{-1}(G))\subset A[x] \ \vert \ (G) \in \MFFunctor{\tilde I_{\geq t}, \tilde J_{\geq t}}(A)\}$
and to every $\kk$-algebra  morphism $\sigma:A\rightarrow A'$, the map 
\begin{eqnarray*}
\MFFunctor{\tilde I_{\geq t}, \tilde J_{\geq t}}^{\tilde g}(\sigma): \MFFunctor{\tilde I_{\geq t}, \tilde J_{\geq t}}^{\tilde g}(A)&\rightarrow \MFFunctor{\tilde I_{\geq t}, \tilde J_{\geq t}}^{\tilde g}(A')\\
\tilde g^{-1}(G) &\mapsto \tilde g^{-1}(\sigma(G)). 
\end{eqnarray*}
%Observe that if $\tilde g^{-1}(G)\in \MFFunctor{\tilde I_{\geq t}, \tilde J_{\geq t}}^{(g)}(A)$, then $G$ belongs to $\MFFunctor{\tilde I_{\geq t}, \tilde J_{\geq t}}(A)$, hence $\sigma(G)$ is in $\MFFunctor{\tilde I_{\geq t}, \tilde J_{\geq t}}(A')$. Since $g \in \mathrm{PGL}_{p(z),\tilde I}\subset \mathrm{PGL}$, $\tilde g$ commutes with $\sigma$ and we obtain that $\sigma(G)$ belongs to $\MFFunctor{\tilde I_{\geq t}, \tilde J_{\geq t}}^{(g)}(A')$.

The transformation $\tilde g^{-1}$ induces a natural isomorphism of functors between the functors $\MFFunctor{\tilde I_{\geq t}, \tilde J_{\geq t}}$  and $\MFFunctor{\tilde I_{\geq t}, \tilde J_{\geq t}}^{\tilde g}$. Hence, $\MFFunctor{\tilde I_{\geq t}, \tilde J_{\geq t}}^{\tilde g}$ is an open subfunctor of $\HFunctor{X}{p(z)}$ for every $g \in \PGL{k+1}$
thanks to Theorem \ref{th:open subfunctor} item \ref{it2:intersectionFunctors}. Analogously, for every $g \in \mathrm{PGL}_{\kk}{(n+1)}$, $\MFFunctor{ \tilde J_{\geq t}}^{g}$ is the open subfunctor of $\HFunctor{\mathbb P^n}{p(z)}$ that we obtain from $\MFFunctor{ \tilde J_{\geq t}}$ by the natural isomorphism induced by $g^{-1}$.

\begin{theorem}\label{thm:ric}
Let $\tilde I\subseteq R$ be a saturated quasi-stable ideal such that $\QR=R/\tilde I$ is a Cohen-Macaulay ring and let $X=\mathrm{Proj}(\QR)$ be the scheme defined by $\tilde I$. Let $p(z)$ be a Hilbert polynomial such that $p(t) \leq p_X(t)$ for $t\gg 0$, $r$ be the Gotzmann number of $p(z)$ and $t:=\max\{\mathrm{reg}(\tilde I), r\}$.
Then, there is the open cover
$$\HFunctor{X}{p(z)} = \bigcup_{g\in \PGL{k+1}} \left(\bigcup_{\tilde J\in Q_{p(z),\tilde I}}\MFFunctor{\tilde I_{\geq t}, \tilde J_{\geq t}}^{\tilde g}\right).$$
\end{theorem}

\begin{proof}
We can apply \cite[Proposition 10.3]{Quot} to the Hilbert functor $\HFunctor{\mathbb P^n}{p(z)}$, obtaining the following open cover
\begin{equation}\label{eq:ricFuncPn}
\HFunctor{\mathbb P^n}{p(z)} =\bigcup_{g\in \mathrm{PGL}_{\kk}(n+1)}\left(\bigcup_{\tilde J\in Q_{p(z)}}  \MFFunctor{ \tilde J_{\geq t}}^{g}\right).
\end{equation}
%where $Q_{p(z)}$ is the set containing the saturated quasi-stable ideals $\tilde J$ such that $\QR/\tilde J$ has Hilbert polynomial $p(z)$.
%and   $\MFFunctor{ \tilde J_{\geq t}}^{\tilde g}$ is the open subfunctor that we obtain from  $\MFFunctor{ \tilde J_{\geq t}}$ by the natural isomorphism induced by $\tilde g^{-1}$.
Since $\HFunctor{X}{p(z)}$ is a closed subfunctor of $\HFunctor{\mathbb P^n}{p(z)}$, we have an open cover of $\HFunctor{X}{p(z)}$ intersecting it with the open subfunctors of \eqref{eq:ricFuncPn}. In order to cover $\HFunctor{X}{p(z)}$  it is enough to consider $J\in Q_{p(z),\tilde I}$, thanks to Theorem \ref{thm:linkToGenMethod}.

We now observe that it is even enough to only take $g\in \PGL{k+1}$ thanks to
Corollary \ref{cor:tranformation}. We can now conclude taking into account that $\MFFunctor{ \tilde J_{\geq t}}^{g}\cap \HFunctor{X}{p(z)}=\MFFunctor{\tilde I_{\geq t}, \tilde J_{\geq t}}^{\tilde g}$ by Theorem \ref{th:open subfunctor} item \ref{it2:intersectionFunctors},  combined with \cite[Exercise VI-11]{EH}.
\end{proof}

\section{Lex-point in Hilbert schemes over Macaulay-Lex quotients on quasi-stable ideals}
\label{sec:lex-point}

\subsection{Generalities}
Following \cite{MP}, the same notation and terminology as in Section~\ref{sec:open cover} are here used also for quotient rings $\QR:=R/M$ when $M$ is any monomial ideal of~$R$.

\begin{definition}(see \cite{MM2010,MP})
A set $W$ of terms of $\QR$ is called a {\em lex-segment of $\QR$} if, for all terms $u,v\in \QR$ of the same degree, if $u$ belongs to $W$ and $v>_{lex}u$ then $v$ belongs to $W$. A monomial ideal $U$ of $\QR$ is called a {\em lex-ideal} if the set of terms in $U$ is a lex-segment of $\QR$.
\end{definition}

\begin{example}\label{ex:lex-ideal}
The image of a lex-ideal of $R$ in $\QR$ is a lex-ideal of $\QR$. However, there are lex-ideals of $\QR$ that are not the image of a lex-ideal of $R$. For example, consider $n=3$ and $\tilde I=(x_3^2,x_2^5)$. Then, the image $U$ in $\QR$ of the ideal $\tilde J=(x_3^2,x_3x_2,x_3x_1^2,x_2^5)\subseteq R$ is a lex-ideal in $\QR$, but $\tilde J$ is not a lex-ideal in $R$. 
\end{example}

The quotient ring $\QR$ is called a {\em Macaulay-Lex ring} if,
for any homogeneous ideal $U$ of $\QR$, there exists a lex-ideal of $\QR$ having the same Hilbert function as $U$ (e.g.~\cite{MM2010}). 
If the monomial ideal $M$ induces a Macaulay-Lex quotient ring, then we say that $M$ is \emph{Macaulay-Lex} and that $M$ is a \emph{Macaulay-Lex} ideal.

\begin{example}\label{example:MacaulayLexMonomialIdeals}
Various families of Macaulay-Lex monomial ideals $M\subseteq R$ are known. We list some of them explicitly and point to references in other cases.
\begin{enumerate}
  \item The most well-known class of Macaulay-Lex ideals are the \emph{Clements-Lindstr\"{o}m ideals} \cite{CL}. They are ideals 
generated by a regular sequence $x_n^{d_n}, x_{n-1}^{d_{n-1}},\dots,x_0^{d_0}$, where $1\leq d_n\leq d_{n-1}\dots\leq d_0$ are integers or $\infty$ with $x_i^{\infty}=0$. A Clements-Lindstr\"om ideal is a quasi-stable ideal $\tilde I$ and the quotient ring $\QR:=R/\tilde I$, which is called a {\em Clements-Lindstr\"{o}m ring}, is Cohen-Macaulay. %If $d_0=\infty$, then $\tilde I$ is a saturated ideal.
If $d_n=1$, one may as well work in a quotient of $\kk[x_0,\ldots,x_{n-1}]$ and drop the generator~$x_n$. 
  \item Abedelfatah \cite[Theorem 4.5]{AbedelfatahMacaulayLexRings} discovered two families of Macaulay-Lex ideals, whose generating sets 
show some similarities to the generators of Clements-Lindstr\"{o}m ideals. In our conventions, they are given as follows, under the conditions
$2\leq e_n\leq e_{n-1}\leq\cdots\leq e_0\leq \infty$ and $t_i<e_i$ for all $i$:
\begin{itemize}
    \item $I=(x_n^{e_n},x_{n}^{t_n}x_{n-1}^{e_{n-1}},\ldots,x_n^{t_n}x_0^{e_0})$,
    \item $I=(x_n^{e_n},x_n^{e_n-1}x_{n-1}^{e_{n-1}},x_n^{e_n-1}x_{n-1}^{t_{n-1}}x_{n-2}^{e_{n-2}},\ldots, x_n^{e_n-1}x_{n-1}^{t_{n-1}}\cdots x_{1}^{t_{1}}x_0^{e_0})$.
\end{itemize}
One can show that every such ideal is quasi-stable.
  \item Mermin \cite{MerminMonomialRegularSequences} showed that a monomial regular sequence generates a Ma\-cau\-lay-Lex ideal if and only if it is of the form 
$$(x_n^{e_n},x_{n-1}^{e_{n-1}},\ldots,x_{r+1}^{e_{r+1}},x_r^{{e_r}-1}x_i),$$
where $e_n\leq e_{n-1}\leq \ldots\leq e_r$ and $i\leq r$. Note that such an ideal is 
quasi-stable if and only if $i=r$, i.e., if it is a Clements-Lindstr\"{o}m ideal.
  \item A complete characterization of all Macaulay-Lex monomial ideals in $\kk[x_0,x_1]$ is known (see e.g. \cite{HeMacaulayLexIdealsIn2Variables}). In particular, there are many quasi-stable Macaulay-Lex ideals in the polynomial ring with two variables.
\item Given $n$ zero-dimensional Macaulay-Lex monomial ideals $M_i$, each of them in a polynomial ring with two variables, 
one can construct \cite{HeMacaulayLexIdealsSomeGeneralizedMacaulayLexIdeals} a zero-dimen\-sional Macaulay-Lex ideal in $R$ from $M_1,\ldots,M_n$.
Being zero-dimen\-sio\-nal, this ideal is also quasi-stable. Note that the construction in \cite{HeMacaulayLexIdealsSomeGeneralizedMacaulayLexIdeals}
covers also more general cases.
  \item For each Macaulay-Lex monomial ideal $M_i\subseteq R_i=\kk[x_i,\ldots,x_n]$, where $i\in \{1,\ldots,n\}$, also the extension ideal
$(M_i)\cdot {R}\subseteq R$ is Macaulay-Lex.~\cite[Thm. 4.1]{MePe2006}
\end{enumerate}
\end{example}

\begin{remark}\label{remark:ML_property_under_ideal_operations}
The property of being Macaulay-Lex is not preserved under many common ideal operations like for example saturation.

Consider the monomial ideal $I=(x_3^2,  x_3x_2^7, x_3x_2x_1^7, x_3x_2x_1^2x_0^7)\subset R=\mathbb{K}[x_0,\ldots,x_3]$. It is Macaulay-Lex  according to Example~\ref{example:MacaulayLexMonomialIdeals} item (2). Its saturation is  $I^{sat}=(x_3^2, x_3x_2x_1^2, x_3x_2^7)$. We can apply \cite[Proposition 1]{HeMacaulayLexIdealsSomeGeneralizedMacaulayLexIdeals} to see that $I^{sat}$ is not Macaulay-Lex. Consider the set  $J_4=\{x_3x_1^3\}\subseteq R/I$. Its lex-segment in $R/I^{sat}$ is  $L_4=\{x_3x_2^3\}$. Multiplying $J_4$ with the generators $x_0,\ldots,x_3$ of the homogeneous maximal ideal, we obtain  the set $\{x_0,x_1,x_2,x_3\}\{x_3x_1^3\}=\{x_3x_1^3x_0, x_3x_1^4, x_3x_2x_1^3, x_3^2x_1^3\}$ which, modulo the ideal $I^{sat}$, is equal to $\{x_3x_1^3x_0, x_3x_1^4\}$, a set with two  elements; but multiplying $L_4$ with $x_0,\ldots,x_3$, we obtain the set $\{x_0,x_1,x_2,x_3\}\{x_3x_2^3\}=\{x_3x_2^3x_0, x_3x_2^3x_1, x_3x_2^4, x_3^2x_2^3\}$ which, modulo the ideal $I^{sat}$, is equal to $\{x_3x_2^4, x_3x_2^3x_1,  x_3x_2^3x_0\}$, with three elements. Thus, it is not possible to find a lex-ideal in the quotient $R/I^{sat}$ having the same Hilbert function as the ideal $(J_4)+I^{sat}/I^{sat}$.

The Macaulay-Lex property is in general also not preserved when truncating an ideal at a given degree.

Consider the ideal $I=(x_3^2, x_3x_2^3, x_3x_2x_1^3, x_3x_2x_1^2x_0^3)\subset \mathbb{K}[x_0,...,x_3]$. It is Macaulay-Lex according to Example~\ref{example:MacaulayLexMonomialIdeals} item (2). Consider the same ideal in $\mathbb{K}[w,x_0,...,x_3]$, where the variable $w$ is ranked lower than the others. The ideal is still Macaulay-Lex, see Example~\ref{example:MacaulayLexMonomialIdeals} item (6). Moreover, it is also saturated in this ring. As mentioned in Example~\ref{example:MacaulayLexMonomialIdeals} item (2), the ideal is quasi-stable. Its regularity is 8, as its minimal Pommaret basis has maximal degree 8 \cite[Corollary 5.5.18]{Seiler:InvolutionBook} (one element of degree $8$ being $x_3x_2^2x_1^2x_0^3$). The truncation $I_{\geq 8}$ is a stable ideal \cite[Proposition 5.5.19]{Seiler:InvolutionBook}. A minimal generator of $I_{\geq 8}$ is $a=x_3x_2x_1^2x_0^4$. Now consider the term $b=x_3x_2^2x_1x_0^3$. $b$ is also of degree $8$, is lexicographically larger than $a$, and $\min(a)=\min(b)$. But $b$ is not in the ideal $I_{\geq 8.}$ This shows that $I_{\geq 8}$ is not piecewise lexsegment \cite[Definition~3.1]{AbedelfatahMacaulayLexRings} although it is stable. It follows \cite[Theorem 3.6]{AbedelfatahMacaulayLexRings} that $I_{\geq 8}$ is not Macaulay-Lex.
\end{remark}

From now we assume that $M:=\tilde I$ is a quasi-stable ideal in~$R$ and that $\QR=R/\tilde I$ is a Macaulay-Lex ring. Recall that a monomial ideal $U$ of $\QR$ is quasi-stable in $\QR$ if the ideal $(\mathcal B_{U}\cup \mathcal B_{\tilde{I}})$ is quasi-stable in $R$.   

\begin{lemma}\label{lemma:lex-ideal}
With the notation above,
\begin{itemize}
\item[(i)] If $W$ is a lex-segment in $\QR$ then $\{x_0,\dots,x_n\}\cdot W$ is a lex-segment in $\QR$.
\item[(ii)] A lex-ideal $U$ in $\QR$ is quasi-stable. 
%If $U$ is a lex-ideal of $\QR$, then the ideal generated by $\mathcal B_U\cup B_{\tilde I}$ in $R$ is quasi-stable.
\item[(iii)] If $U$ is a lex-ideal of $\QR$, then $(\mathcal B_U\cup \mathcal B_{\tilde I})^{sat}/\tilde I$ is a lex-ideal.
\end{itemize}
\end{lemma}

\begin{proof}
For item (i) see \cite[Proposition 2.5]{MePe2006}.
For item (ii), let $\tau$ be a term of $U$ with minimal variable $x_i$ and let $x_j>x_i$. Since $x_j\frac{\tau}{x_i}>_{lex} \tau$, we must have that $x_j\frac{\tau}{x_i}$ belongs to $U$ unless it belongs to $\tilde I$. Then, we conclude because $\tilde I$ is quasi-stable. Item (iii) now follows from item (ii) and from the properties of the lexicographic term order, because  thanks to the properties of quasi-stable ideals we obtain the saturation replacing $x_0$ by~$1$ in every generator.
\end{proof}

If $\tilde I$ is a saturated ideal   we can consider  the projective scheme $X=\mathrm{Proj}(\QR)$ and   the Hilbert scheme $\HScheme{X}{p(z)}$ on the Macaulay-Lex ring $\QR$, for an admissible Hilbert polynomial $p(z)$. 

\begin{theorem}\label{th:lex-point}
Let $\QR=R/\tilde I$ be a Macaulay-Lex ring  and $X=\mathrm{Proj}(\QR)$. Then
$\HScheme{X}{p(z)}$ is non-empty if and only if it contains a (unique) point $Y$ defined by a lex-ideal of $\QR$. 
Moreover, $Y$ has the minimal possible Hilbert function in $\HScheme{X}{p(z)}$.
\end{theorem}

\begin{proof}
Let $r$ be the maximum between the Gotzmann number of $p(z)$ and the regularity of~$\tilde I$. 
If $\HScheme{X}{p(z)}$ is non-empty, then there exists at least a lex-ideal $U$ of $\QR$ such that $\QR/U$ has Hilbert polynomial $p(z)$, because $\QR$ is Macaulay-Lex. 
In particular, letting $\tilde p(z)$ be the Hilbert polynomial of $\QR=R/\tilde I$, the set $W$ made of the $\tilde p(r)-p(r)$  lex-largest terms of $U$ of degree $r$ is a lex-segment. Thanks to Lemma \ref{lemma:lex-ideal} item (iii), the ideal $(W\cup \mathcal B_{\tilde I})^{sat}/\tilde I$ is a lex-ideal too and by construction defines the desired point $Y$ in $\HScheme{X}{p(z)}$. 
Indeed, by definition of saturation we obtain the same point $Y$ starting from any other lex-ideal of $\QR$ with Hilbert polynomial $p(z)$, so that the last assertion also follows.
\end{proof}

\begin{definition}\label{def:lex-point}
Let $\QR$ be a Macaulay-Lex ring over a saturated quasi-stable ideal $\tilde I$ and $X=\mathrm{Proj}(\QR)$. If $\HScheme{X}{p(z)}$ is non-empty, then its unique point defined by a lex-ideal of $\QR$ is called the {\em lex-point} of $\HScheme{X}{p(z)}$.
\end{definition}

%\begin{remark}\label{rem:lex-point}
%With the notation above, let $r$ be the maximum between the Gotzmann number of $p(z)$ and the regularity of~$\tilde I$. If $W$ is the lex-segment of $\QR$ made of $\tilde p(r)-p(r)$ terms of degree $r$ and $\vert \{x_0,\dots,x_n\}\cdot W \cup \mathcal B_{{\tilde I}_{r+1}}\vert =p(r+1)$, then the ideal generated by $W$ is a lex-ideal in $\QR$ by Lemma \ref{lemma:lex-ideal} item (i) and defines the lex-point of $\HScheme{X}{p(z)}$, thanks to Gotzmann's persistence theorem.
%\end{remark} 

\subsection{Examples of smooth and singular lex-points}

We now give some examples both of smooth and singular lex-points in Hilbert schemes over Macaulay-Lex rings. %Our experiments highlight that it seems more difficult to find smooth lex-points than singular ones. 

When it is necessary, we apply Algorithm~\ref{algorithm} and, like in Section \ref{sec:relative marked functor}, take into account the fact that a scheme $\MFScheme{\tilde I_{\geq t}, \tilde J_{\geq t}}$ is an open subscheme of $\HScheme{X}{p(z)}$. Hence, the Zariski tangent space to $\MFScheme{\tilde I_{\geq t}, \tilde J_{\geq t}}$ at one of its points is equal to the Zariski tangent space to $\HScheme{X}{p(z)}$ at the same point (see also \cite[Corollary 1.9]{Gore}). 

\begin{example}\label{ex:trivial}
Here is a trivial case of smooth lex-points.
When $\tilde I=(x_3,\dots,x_n)$, $X$ is the projective plane and so every point of the Hilbert scheme $\HScheme{X}{p(z)}$ is smooth, for a constant polynomial $p(z)$. Hence, in this case every point in $\HScheme{\mathbb P^n_{\kk}}{p(z)}$, even singular, corresponds to a smooth point in $\HScheme{X}{p(z)}$.
This is the extremal case of the trivial situation in which $\tilde I$ is generated by variables. In fact, in such case the image in $\QR$ of a lex-point of $R$ is still smooth because it is simply the lex-point in a Hilbert scheme over a lower dimensional projective space. 
\end{example}

\begin{example}\label{ex:smooth classes NoCL}
The following saturated quasi-stable ideals $\tilde I\subset\tilde J$ in the ring $R=\mathbb \kk[x_0,\dots,x_n]$ give {\em smooth} lex-points $Y$ in the Hilbert scheme $\HScheme{X}{p(z)}$, where $X=\mathrm{Proj}(R/\tilde I)$, $p(z)$ is the Hilbert polynomial of $R/\tilde J$ and $Y$ is defined by $\tilde J/\tilde I$:
\begin{itemize}
\item[(i)] $\tilde I=(x_n^k,x_n^{k-1}x_{n-1}) \subset \tilde J=(x_n^k,x_n^{k-1} x_{n-1},x_n^{k-1} x_{n-2})$, for every $n\geq 3$ and $k\geq 2$
\item[(ii)] $\tilde I=(x_n,x_{n-1})^2\subset \tilde J=(x_n^2,x_nx_{n-1},x_nx_{n-2},x_{n-1}^2)$, for every $n\geq 3$.
%\item[(iii)] $\tilde I=(x_n^2,x_nx_{n-1},x_nx_{n-2})\subset \tilde J=(x_n^2,x_nx_{n-1},x_nx_{n-2},x_{n-1}^2)$, for every $n\geq 3$. Non so se $\tilde I$ è Macaulay-Lex.
\end{itemize}
In case (i), the ideal $\tilde I$ is Macaulay-Lex thanks to \cite[Theorem~4]{HeMacaulayLexIdealsIn2Variables} and \cite[Theorem~4.1]{MePe2006} and the ideal $\tilde J$ already defines a lex-point in $R$. By a pencil-and-paper work, we apply Algorithm~\ref{algorithm} and obtain the following results. 
The set $\mathscr H$ is made of the polynomial $h= x_n^{k-1} x_{n-2} + c_1 x_0^k+ \dots c_{s}x_n^{k-1}x_{n-3}$ in the ring $\mathbb Q[c_1,\dots,c_s][x_0,\dots,x_n]$, with $s=\binom{n+k}{n}-3$.
By $\rid{\mathscr H^\ast }$ we need to reduce the three polynomials $x_n h$, $x_{n-1} h$, $x_n\cdot x_n^{k-1}x_{n-1}$ and then apply the reduction modulo $\tilde I$. The last polynomial $x_n\cdot x_n^{k-1}x_{n-1}$ is already reduced with respect to  $\rid{\mathscr H^\ast }$ and moreover belongs to $\tilde I$.

The terms in the polynomial $x_{n-1}h$ that do not belong to $\tilde I$ are already reduced modulo $\rid{\mathscr H^\ast }$. Hence, their coefficient must be null. On the other hand, the coefficients of the terms belonging to $\tilde I$ are free. Such terms are of type $x_{n-1}\cdot x_n^{k-1}x_j$, where $0 \leq j \leq n-3$, hence they are $n-2$.

The terms in the polynomial $x_nh$ that belong to $\tilde I$ are 
of type $x_n\cdot x_n^{k-1}x_j$, where $0 \leq j \leq n-3$, hence they have the same coefficients of the analogous terms in  $x_{n-1}h$. All the other terms must have a null coefficient because of the polynomial $x_{n-1}h$, even those that are not reduced with respect to $\rid{\mathscr H^\ast }$.

In conclusion, the ideal $\mathscr R'$ is generated by $s-(n-2)$ parameters $c_i$, and hence $\MFScheme{\tilde I_{\geq t}, \tilde J_{\geq t}}$ is a linear variety of dimension $n-2$. Hence, it is smooth and, in particular, the point $Y$ is smooth in $\HScheme{X}{p(z)}$. 

In case (ii), the ideal $\tilde I$ is Macaulay-Lex thanks to \cite[Theorem~2.1]{MerminMuraiBettinNumbersOfLexIdealsOverMLRgs} and the ideal $\tilde J$ does not define a lex-point in $R$, but $\tilde J/\tilde I$ defines a lex point in $\QR=R/\tilde I$. As for case (i), we apply Algorithm~\ref{algorithm}. We proceed in an analogous way as for case (i), obtaining that $\MFScheme{\tilde I_{\geq t}, \tilde J_{\geq t}}$ is a linear variety of dimension $2n-3$.
\end{example}

\begin{example}\label{ex:singular no CL}
The saturated quasi-stable ideals $\tilde I=(x_3^3,x_3^2x_2)\subset\tilde J=(x_3^2,x_3x_2,$ $x_3x_1)\subset R=\mathbb \kk[x_0,\dots,x_3]$ give a {\em singular} lex-point $Y$ in the Hilbert scheme $\HScheme{X}{p(z)}$, where $X=\mathrm{Proj}(R/\tilde I)$, $p(z)$ is the Hilbert polynomial of $R/\tilde J$ and $Y$ is defined by $\tilde J/\tilde I$. Like in Example \ref{ex:smooth classes NoCL}, the ideal $\tilde I$ is Macaulay-Lex thanks to \cite[Theorem~4]{HeMacaulayLexIdealsIn2Variables} and \cite[Theorem~4.1]{MePe2006}.
For $t\geq 1$, the dimension of $\MFScheme{\tilde I_{\geq t}, \tilde J_{\geq t}}$ is $2$ and the dimension of its Zariski tangent space at $Y$ is $6$. In this case we have $\mathcal P_{\tilde J}\setminus \mathcal P_{\tilde I}=\mathcal P_{\tilde J}$ and so we apply the computation described at the end of Section \ref{sec:MFandHilb}.
\end{example}

\begin{remark}\label{rem:smooth no-lex CL}
%The saturated quasi-stable ideals $\tilde I=(x_3^3)\subset\tilde J=(x_3^3,x_3^2x_2,x_3^2x_2)\subset R=\mathbb \kk[x_0,\dots,x_3]$ give a {\em singular} lex-point $Y$ in the Hilbert scheme $\HScheme{X}{p(z)}$, where $X=\mathrm{Proj}(R/\tilde I)$, $p(z)$ is the Hilbert polynomial of $R/\tilde J$ and $Y$ is defined by $\tilde J/\tilde I$. Indeed, for $t\geq 2$ the dimension of $\MFScheme{\tilde I_{\geq t}, \tilde J_{\geq t}}$ is $2$ and the dimension of its Zariski tangent space at $Y$ is $5$.
The points $Y$ of Examples \ref{ex:first example} and~\ref{ex:second example} are singular lex-points in Hilbert schemes over Clements-Lindstr\"om rings (see also Example \ref{ex:lex-ideal}). Actually, it seems not obvious to find a non-trivial example of smooth lex-point in a Clements-Lindstr\"om ring. However, there are other points that are smooth. For example, letting $\tilde I=(x_3^2)\subset\tilde J=(x_3^2,x_{2}^2)\subset R=\mathbb \kk[x_0,\dots,x_3]$, both the dimensions of $\MFScheme{\tilde I_{\geq 1}, \tilde J_{\geq 1}}$ and of the Zariski tangent space at the point defined by $\tilde J/\tilde I$, which is not a lex-point, are $8$ in the Hilbert scheme $\HScheme{X}{p(z)}$ with $X=\mathrm{Proj}(R/\tilde I)$ and $p(z)=4z$.
\end{remark}

\section*{Acknowledgments}
The first and second authors are members of GNSAGA (INdAM, Italy).

\bibliographystyle{amsplain}
\bibliography{RelativeMarkedBases}

\end{document}